\documentclass[final]{elsarticle}

\usepackage{amscd}
\usepackage{amssymb}
\usepackage{amsmath}
\usepackage{theorem}
\usepackage{vmargin}
\usepackage{color}
\usepackage{tikz}
\usetikzlibrary{matrix}

\newtheorem{lemma}{Lemma}
\newtheorem{corollary}{Corollary}

\newtheorem{theorem}{Theorem}
 
\newtheorem{proposition}{Proposition} 
\newtheorem{remark}{Remark}
\newtheorem{example}{Example}

\newenvironment{proof}{\textbf{Proof.}}{\hfill $\Box$}

\newcommand{\trans}{\mathsf{T}}
\newcommand {\N} {\mathbb{N}}
\newcommand{\di}[1]{\operatorname{d}\!#1}
\newcommand{\p}{\mathbf{P}}
\newcommand {\e} {\varepsilon}
\newcommand {\C} {\mathbb{C}}
\newcommand{\R}{\mathbb{R}}

\journal{}

\begin{document}

\begin{frontmatter}

\title{Uniform and $L^q$-Ensemble Reachability of Parameter-dependent Linear Systems}
\author[First]{ Gunther Dirr and Michael Sch\"onlein}

\address[First]{Institute for Mathematics, University of W\"urzburg, Emil-Fischer Stra\ss e 40, 97074 W\"urzburg, Germany \\
(e-mail: \{dirr,schoenlein\}@mathematik.uni-wuerzburg.de).}

%
%
%
%

\begin{abstract}
In this paper, we consider families of linear systems (linear ensembles) defined by matrix pairs 
$\big( A(\theta),B(\theta) \big)$ depending on a parameter $\theta \in \p$ that is varying
over a compact subset $\p$ of the complex plane. In particular, we investigate the following control
task: Find an open-loop control which is {\it independent} of the parameter $\theta \in \p$ and
steers a given family of initial states $x_0(\theta)$ arbitrarily close to a desired family of 
terminal states $f(\theta)$ in finite time. Here, the maps $\theta \mapsto x_0(\theta)$ and
$\theta \mapsto f(\theta)$ are assumed to lie in a common appropriately chosen {Banach space 
$X_n(\p)$ of $\C^n$-valued functions}. If this task is solvable for all initial and terminal states, the pair 
$\big( A(\theta),B(\theta) \big)$ is called {(completely)} ensemble controllable with respect to $X_n(\p)$. 

{Using a well-known infinite-dimensional version of the Kalman rank condition for systems on Banach
  spaces, we derive sufficient conditions for cascade and parallel connections linear ensembles. Moreover, we prove an
  abstract decomposition theorem which results from a spectral splitting of the matrix family $A(\theta)$.
  Based on thses findings as well as approximation theory and cyclicity conditions of multiplications operators, we obtain
  necessary and sufficient conditions for ensemble controllability (reachability) with respect to the Banach spaces of
  continuous functions and $L^q$-functions. In the last section, results on {averaged} controllability (reachability)
  for linear families $\big( A(\theta),B(\theta),C(\theta) \big)$ are presented.}
\end{abstract}

\begin{keyword}
parameter-dependent systems \sep ensemble reachability \sep infinite-dimensional systems \sep 
\MSC[2010] 30E10  \sep 47A16  \sep 93B05
\end{keyword}

\end{frontmatter}

\section{Introduction}\label{intro}

In recent years the task of controlling a large, potentially infinite, number of states or systems
at once using only a single open-loop input or a single feedback controller has posed a challenge
in mathematical systems and control theory. Nowadays the term {\it ensemble control} has been established
to refer to this {area of research, cf.~\cite[Section~2.4]{brockett2012notes}, but
  {\it simultaneous control} or {\it control of families of systems} are also common, see
\cite{buscain_2012,ghosh2000sufficient,hautussontag1986,loheac:hal-01164525,sontagwang1990}. Closely related topics
  are robust control \cite{amato2006} and the blending problem as considered in \cite{Tannenbaum1981invariance}.}


Of course, there are a lot{s} different scenarios which require to control a large or even an infinite
number of states: (i) First, think of a system which is composed by a tremendous number of 
subsystems, like a flock of birds or a swarm of bees \cite{brockett2010control}. (ii) An other reason 
for a huge state space could result from uncertainties in the initial data. For instance, if only 
a probability distribution of the initial states is known then the ensemble control problem leads
to a transport problem in terms of density functions and therefore to controllability and observability
issues of the Liouville and the Fokker-Plank equation
\cite{brockett2012notes,chen2017optimal,fleig2016estimates,shen2017discrete,zeng2016ensemble}.
(iii) A third setting arises again from uncertainties but now with respect to the model parameters.
In this case, the system depends on parameters and the goal is to achieve a control task by using 
only a single or a few open-loop inputs which are {\it independent} of the (usually unkonwn) 
model parameters \cite{li2009ensemble}. In this paper, we focus on linear systems which arise from 
scenario (iii).

For recent contributions to the linear ensemble control problem we refer to, e.g.~\cite{li2011} and
\cite{li_tac_2016}. For nonlinear parameter-dependent systems results have been obtained in
\cite{Beauchard-Coron-Rouchon_2010,li2009ensemble} and \cite{Agrachev_ensemble_lie_2016}. 
A monemt-based approach to ensemble control, in particular for linear systems, is considered in 
\cite{loheac:hal-01164525} and \cite{Zeng_scl_2016}. Another aspect within this context is to steer the 
average of the ensemble states towards a desired terminal value, cf.~\cite{Zuazua_jde_2017,zuazua2014averaged}. 
{This will be discussed in Section~\ref{sec:output}}. We also point out the work \cite{Gauthier_scl_2018}
that investigates asymptotic ensemble stabilizability. Moreover, we notice that all of the contributions
mentioned above treat continuous-time systems.

\bigskip
\noindent
In the sequel, we introduce our precise model which consists of a family of parameter-dependent linear
control systems 
\begin{align}\label{eq:sys-par} 
\frac{\partial x}{\partial t}
(t,\theta)=A(\theta)x(t,\theta)+B(\theta)u(t), \quad x(0,\theta)=x_0(\theta)\in\C^n\,.
\end{align}
{While the matrix $A(\theta)\in\mathbb{C}^{n\times n}$ is assumed to depend continuously on
the parameter $\theta \in \p$, which varies over a nonempty compact set $\p \subset \C$, the regularity of 
the input matrix $B(\theta)\in \mathbb{C}^{n\times m}$ is subject to the particular case under consideration
and will  be specified later. Controls $u(t)$, however, are independent of $\theta$ and complex-valued
throughout the paper. The latter assumption allows complex state spaces and thus simplifies later spectral
analysis. Nevertheles, the ``real case'' is covered by this approach as well, cf.~Lemma \ref{lem:complexification}.}
Moreover, we emphasize that the results of this paper also apply to discrete-time systems 
\begin{align}\label{eq:sys-par-discrete} 
x(k+1,\theta)=A(\theta)x(k,\theta)+B(\theta)u(k), \quad x(0,\theta)=x_0(\theta) \in\C^n.
\end{align}
For $u\in L^1([0,T],\mathbb{C}^m)$ or $u{=(u(0),...,u(T-1) ) \in  \C^m \times \cdots \times \C^m }$,
let $\varphi\big(T,u,x_0(\theta),\theta\big)$ denote the solution to \eqref{eq:sys-par} or \eqref{eq:sys-par-discrete},
i.e.
\begin{equation*}
\varphi\big(T,u,x_0(\theta),\theta\big) = {\rm e}^{T A(\theta)}x_0(\theta) + 
\int_0^T {\rm e}^{(T-\tau)A(\theta)}B(\theta)u(\tau)\di{\tau}
\end{equation*}
or
\begin{equation*}
\varphi\big(T,u,x_0(\theta),\theta\big) = \big(A(\theta)\big)^T x_0(\theta) + 
\sum_{k=0}^{T-1} \big(A(\theta)\big)^{T-1-k}B(\theta)u(k)\,.
\end{equation*}
To handle the continuous and discrete-time case at the same time we will denote for $T\geq 0$ the set of inputs
by $U(T)$. That is, in continuous-time one has $U(T):=L^1([0,T],\mathbb{C}^m)$ and in discrete-time
$U(T):=\C^m \times \cdots \times \C^m$. From now on $T\geq 0$ is either a nonnegative real or a natural number
depending on the system under consideration.

The central problem can be sketched as follows: Given a family of initial states $x_0(\theta)$ and
a family of terminal states $f(\theta)$. {Find} { $T\geq 0$ and $u \in U(T)$} such that  
\begin{align*}
\varphi(T,u,x_0(\theta),\theta) = f(\theta) \qquad \text{ for all } \, \theta \in \p.
\end{align*}
{In linear control theory, this property is usually called \textit{(complete) controllability},
cf.~e.g.~\cite[Section~3.2]{trentel_book_2001}. Note that for finite $\p =\{\theta_1,....,\theta_N\}$ the ensemble
control problem reduces to the classical problem of controlling a parallel connection of finitely many linear
systems. In this case the state space remains finite-dimensional and the situation is well-understood,
cf.~\cite{fuhrmann2015mathematics}. Therefore, we assume here and henceforth that $\p$ is infinite.}
{The key point in ensemble control} is---as mentioned above---that the input $u$ has to be
independent of the system parameter $\theta \in \p$. Without this crucial requirement the controllability
analysis of \eqref{eq:sys-par} and \eqref{eq:sys-par-discrete} would be much simpler, cf.~\cite{curtain_tac_2015}.
{But due to its special structure this problem is so far not covered by} standard
textbooks on infinite dimensional linear {systems} such as \cite{curtainzwart95,fuhrmann_hilbert_book}.
{All what is known from infinite dimensional systems theory is} that for infinite parameter spaces
$\p$ both equations \eqref{eq:sys-par} and \eqref{eq:sys-par-discrete}, are never (completely) controllable, 
cf. \cite[Theorem~3.1.1]{triggiani75} and \cite[p.~244]{fuhrmann_hilbert_book}. Hence, only weaker notions
of controllability are reasonable {and thus} we will focus on an approximate version of the above 
controllability concept.

For making this statement precise we have to fix some notation that will be used {in the
sequel. Let $X(\p)$ denote an arbitrary separable Banach spaces of functions defined on $\p$ with
values in $\C$ and let $X_{n,m}(\p)$ consist of all $(n \times m)$-matrices with entries in $X(\p)$.
Furthermore, set $X_n(\p) :=  X_{n,1}(\p)$. Thus $X_n(\p)$ is simply the $n$-fold cartesian product of
$X(\p)$ and therefore again a Banach space\footnote{In general, we assume that $X_n(\p)$ is equipped
with the maximum norm, i.e.~$\|x\|_{X_n(\p)} := \max_{1 \leq i \leq n}\|x_i\|_{X(\p)}$. However, for
the special case $X_n(\p) := L^q_n(\p)$ the corresponding {$L^q$-construction lends}
itself as better choice.}. In particular, the corresponding spaces of continuous and
$L^q$-functions will be denoted by $C_n(\p)$, $C_{n,m}(\p)$ and $L^q_n(\p)$, $L^q_{n,m}(\p)$, respectively.
Moreover, for fixed $A \in C_{n,n}(\p)$ we always assume that the induced multiplication
operator ${\cal M}_A: X_n(\p) \to X_n(\p)$ given by
\begin{align}
\label{eq:def-operators-scalar}
{\cal M}_Af(\theta) := A(\theta)f(\theta)
\end{align}
is well-defined and bounded. If this holds for all $A \in C_{n,n}(\p)$ with
\begin{align}
\label{eq:def-operator-norm-scalar}
  \|{\cal M}_A\| \leq \kappa \, \max_{\theta \in \p \atop 1 \leq i,j \leq n} |a_{ij}(\theta)|\, ,
\end{align}
for some constant $\kappa > 0$ independent of $A$ then $X_n(\p)$ or actually $X(\p)$ is called a
{\it multiplication space} (or for short M-space). Standard M-spaces are for instance $C(\p)$ and
$L^q(\p)$ (with $\kappa =1$ for $n=1$). Important examples of non-M-spaces are provided by Hardy
spaces $H^p(\mathbb{D})$ or the disc algebra $\mathcal{A}(\mathbb{D})$ (cf.~Example \ref{example:non-sep}).
Finally, for any $B \in X_{n,m}(\p)$ the input operator ${\cal M}_B:\C^m \to X_n(\p)$,
${\cal M}_Bv(\theta) = B(\theta)v$ is also well-defined and bounded as its domain $\C^m$ is finite dimensional.} 

By means of the above definitions, \eqref{eq:sys-par} and \eqref{eq:sys-par-discrete}
are equivalent to the (infinite dimensional) linear control systems
\begin{align}
\begin{split}
\label{eq:sys-inf}
\dot x (t)&={\cal M}_Ax(t) + {\cal M}_Bu(t),\quad x(0) =x_0 \in X_n(\p)
\end{split}
\end{align}
and correspondingly
\begin{align}
\begin{split}
\label{eq:sys-inf-discrete}
x({k}+1)&={\cal M}_Ax({k}) + {\cal M}_Bu({k}),\quad x(0) =x_0 \in X_n(\p)\,.
\end{split}
\end{align}
Here and henceforth, we assume that {$X(\p)$ is an infinite dimensional M-Banach space and 
introduce the notation $\varphi(T,u,x_0)(\theta):= \varphi(T,u,x_0(\theta),\theta)$ to express that 
solutions of \eqref{eq:sys-par} and \eqref{eq:sys-par-discrete} are regarded as} functions of $\theta \in \p$.
Moreover, we call
\begin{align*}
R(A,B) := \left\{ \varphi(t,u,0) \, |\, \,t\geq 0, \, \,  u \in{U(t)} \, \right\}
\end{align*}
the reachability set (from zero).

Now the central notions of this paper read as follows: A pair $(A,B) \in C_{n,n}(\p)\times X_{n,m}(\p)$
is called {\it ensemble reachable (from zero)} with respect to $X_n(\p)$, if for all $f \in X_n(\p)$
and $\e >0$ there {exist} $T \geq 0$ and an input $u \in {U(T)}$ such that
\begin{align}\label{eq:openloop}
\|\varphi(T,u,0) -f\|_{X_n(\p)} < \e.
\end{align}
Similarly, a pair $(A,B) \in C_{n,n}(\p)\times X_{n,m}(\p)$  is called {\it ensemble controllable (to zero)}
with respect to $X_n(\p)$, if for all $x_0 \in X_n(\p)$ and $\e >0$ there {exist} $T \geq 0$
and an input $u \in {U(T)} $ such that
\begin{align}\label{eq:openloop_to_zero}
\|\varphi(T,u,x_0)\|_{X_n(\p)} < \e.
\end{align}
With the above identification {in hands}, ensemble reachability (from zero) / ensemble controllability 
(to zero) of $(A,B) \in C_{n,n}(\p)\times X_{n,m}(\p)$ on $X_n(\p)$ is equivalent to \textit{approximate
reachablility} (from zero) / \textit{approximate controllability} (to zero) of the infinite-dimensional
linear system $({\cal M}_A,{\cal M}_B)$ on $X_n(\p)$. Moreover, for continuous-time linear systems
\begin{align}
\begin{split}
\label{eq:sys-inf-general}
\dot x (t)&={\cal A}x(t) + {\cal B}u(t)\,,
\end{split}
\end{align}
{where ${\cal A}$ and ${\cal B}$ are bounded operators on Banach spaces, one has the
well-known identity
\begin{align}\label{eq:R_T}
{\overline {R(A,B)} = \overline {R_T(A,B)} :=\overline {\left\{ \varphi(T,u,0) \, |\, \, u\in  U(T) \right\}}\qquad \text{for every }T>0}.
\end{align}
as well as the equivalence of the following assertions \cite[Theorem~3.1.1, Remark~3.1.2]{triggiani75}:
\begin{enumerate}
\item[(a)]
System \eqref{eq:sys-inf-general} is approximately reachable (from zero).
\item[(b)]
System \eqref{eq:sys-inf-general} is approximately controllable (to zero).
\item[(c)]
System \eqref{eq:sys-inf-general} is approximately completely controllable, i.e.
for every $\e>0$ and for every pair $x_0,x_1$ there are $T\geq0$ and $u \in U(T)$ such that
\begin{align*}
 \| x_1 - \varphi(T,u,x_0)\|_{X_n(\p)} < \e. 
\end{align*}
\item[(d)]
There exists $T \geq 0$ such that system \eqref{eq:sys-inf-general} is approximately
completely controllable on $[0,T]$, i.e. for every $\e>0$ and for every pair $x_0,x_1$ there exists $u \in U(T)$ such that
\begin{align*}
 \| x_1 - \varphi(T,u,x_0)\|_{X_n(\p)} < \e. 
\end{align*}
\item[(e)]
For all $T > 0$, system \eqref{eq:sys-inf-general} is approximately completely controllable on $[0,T]$.
\end{enumerate}
However, in the discrete-time case Eq.~\eqref{eq:R_T} fails in general even for arbitrary large $T$.
  Moreover, the notion of approximate reachability (from zero) and approximate controllability (to zero) are independent
  of each other, cf.~\cite[Lemma~4.1]{fuhrmann1972weak} and none of both implies approximate (complete) controllability.
This hinders us to treat the continuous-time and discrete-time case fully simultaneously.
But in both cases reachability (from zero) is characterised by Eq.~\eqref{eq:density_cond_er} and thus
  can be studied along  the same lines. Therefore, we focus on reachability (from zero) and drop the addition ``from zero''
  from now on.
  In the Sections~\ref{sec:unif} and~\ref{sec:L_q}, we pay special attention to the Banach spaces of continuous functions $C_n(\p)$
  and integrable functions $L^q_n(\p)$. Hence, there we use the terms {\it uniform ensemble reachability} and {\it $L^q$-ensemble
    reachability}, respectively, to explicitly express the underlying function spaces. Finally, we point out that for continuous-time
  systems the concept of {\it approximate simultaneous controllability} as defined in \cite{loheac:hal-01164525} is equivalent
  to approximate complete controllability and therefore to approximate reachability, whenever ${\cal A}$ and ${\cal B}$
  are bounded operators acting on some Banach spaces.}

\subsubsection*{{Main contributions}} 

{In Section~\ref{sec:prelim}, we provide sufficient condition for the general ensemble reachability
  problem. Our findings can be devided into two catagories: (i) structural results for cascade and parallel connections
  and (ii) a decompostion method based on a spectral splitting of the parameter dependent matrix family $A(\theta)$.
  Results of the first category exploit the particular structure of the matrix pair $(A(\theta),B(\theta))$ to reduced
  the reachability problem to several smaller and easier to solve problems: in the cascade scenario it is obviously
  sufficient to analyze the individual subsystems on the diagonal (cf.~Proposition \ref{prop:technical_MATCOM}); the
  parallel connection is more difficult to handle and requires an additional subtle topological condition (cf.~Theorem
  \ref{thm:parallel_multi}). The main contribuation of the second category (cf.~Theorem~\ref{thm:decomposition}) is based
  on a suitable spectral decomposition of the matrix-multiplication operator induced by $A(\theta)$ which allows to
  transform a given pair $(A(\theta),B(\theta))$ into a block-diagonal form such that the reachability analysis reduces
  again to a block diagonal structure. Both results of catagoriy (i) are in principle valid for general linear systems
  on Banach spaces.}

{In Section \ref{sec:unif} and \ref{sec:L_q}, the above described methods are applied to derive criteria
  for uniform ensemble reachability and $L^q$-ensemble reachability, respectively. More precisely, we provide complete
  characterizations for scalar pairs $(a(\theta),b(\theta))$ to be uniformly and $L^q$-ensemble reachable, cf.~Propositions
  \ref{prop:scalar} and \ref{thm:Lq_scalar}, respectively, and obtain necessary as well as sufficient conditions uniform
  ensemble reachability (Theorems~\ref{thm:suff-charakteristic},~\ref{thm:CC-SI-sufficient-conditions} and~\ref{thm:CC-MI-sufficient-conditions}) and for $L^q$-ensemble reachability (Theorem~\ref{thm:Lq_general} and Proposition~\ref{prop:Lq-MI-sufficient-conditions}).}
Moreover, we consider averaged reachability and apply our obtained results to get {pointwise testable} sufficient conditions for a triple $(A,B,C)$ to be averaged reachable. This is presented in Section~\ref{sec:output}.
\subsubsection*{Further Notation and Definitions}

For a matrix $A\in \C^{n\times m}$, we denote the complex conjugate by $A^\dagger:=\overline{A}^\trans$ 
and its kernel by $\operatorname{ker} A$. For $\Omega \subset \C$ we say that $\Omega$  does not
{\it separate the plane} if $\C\setminus \Omega$ is connected. 
{Furthermore, let $\Omega^\circ$ denote the {\it interior} of $\Omega$ and
 $\overline{\Omega}$ its {\it closure}. We say that a set $C$ is {\it properly contained}
in $\Omega$ if $C \subset \Omega^\circ$ holds}. A compact connected set in the complex plane
containing more than one point is called a {\it continuum}. A set $\Omega$ is {\it locally connected}
if for every $\omega \in \Omega$ and each neighborhood $U{\subset \Omega}$ of $\omega$
there exists a connected neighborhood $V{\subset \Omega}$ of $\omega$ that is contained in $U$.
A set $\Omega$ is called {\it contractible} if the identity map
on $\Omega$ is homotopic to a constant mapping, i.e.~for some $p \in \Omega$ there is a continuous map
$F\colon[0,1]\times \Omega \to \Omega$ such that $F(0,\omega)=\omega$ and $F(1,\omega)=p$ for all 
$\omega \in \Omega$. A $C^k$-{\it path} is a $k$-times continuously differentiable map of a compact
interval into $\C$ and a {\it Jordan curve} {is a homeomorphic image (within $\C$)
 of the unit circle $\partial \mathbb{D}$}. For simplicity, a $C^0$-path will be called simply a {\it path}.
{Morover, $\operatorname{card} \Omega$ stands for the cardinality of $\Omega$ and
  $F\colon U \rightsquigarrow V$ denotes a {\it set-valued map} from $U$ to the power set of $V$. As usual
  for set-valued maps, the image of $F$ is defined by  $F(U) := \bigcup_{x \in U}F(x) \subset V$.}

\section{Structural results}\label{sec:prelim}

We shortly recap some relevant results on approximate reachability.
{Obviously, in terms of the reachable set approximate reachability holds if and only if
$\overline{R(A,B)} = X_n(\p)$. In \cite[Theorem~3.1.1]{triggiani75} it is shown that
approximate reachability of \eqref{eq:sys-inf-general} 
is equivalent to the density condition
\begin{align}\label{eq:density_cond_er}
\overline{\sum_{k\in \mathbb{N}_0}\operatorname{im} \mathcal{A}^k\mathcal{B}} = X_n(\p) .
\end{align}
Taking into account that $\mathcal{A} := \mathcal{M}_A$ and $\mathcal{B} := \mathcal{M}_B$ are multiplication
operators, 
the latter density condition can be written as follows. Let $b_1(\theta), \ldots,b_m(\theta)$ denote
the columns of $B(\theta)$ and let $A^kb_j$ shortly denote the continuous functions 
$\theta \mapsto A(\theta)^kb_j(\theta)$ for {  $k=0,1,2,...$ and $ j=1,...,m$.} Then, a pair 
$(A,B) \in C_{n,n}(\p) \times X_{n,m}(\p)$ is ensemble reachable on $X_n(\p)$ if and only if the set 
\[
{L(A,B)}:=\operatorname{span}\{A^kb_j \;|\; 1 \leq j \leq m,\,\, k\in \mathbb{N}_0\}
\]
is dense in $X_n(\p)$.
Note that discrete-time systems are not considered in \cite{triggiani75}, but condition \eqref{eq:density_cond_er}
as well as the latter equivalence also hold for discrete-time parameter-dependent systems, cf.~\cite[Theorem~1]{schonlein2016controllability}.} 
%
%
Moreover, we emphasize that ${L(A,B)}$ is dense in $X_n(\p)$ if and only if for each $\e>0$ and
each $f \in X_n(\p)$ there exist {complex polynomials} $p_1, \ldots, p_m$ such that 
\begin{align}\label{eq:con-ensemble}
 \left\| \sum_{j=1}^{m}p_j(A)\,b_j - f \right\|_{X_n(\p)} < \e.
\end{align}
The latter condition links {ensemble reachability to polynomial approximations as well as to}
the notion of cyclicity of the multiplication operators ${\cal M}_A$. More precisely, a bounded linear
operator $T$ defined on a separable Banach space $X$ is called $l$-multicyclic if there is an $l$-tuple
$(x_1,...,x_l) \in X^l$ such that the closure of 
\begin{align*}
  {
 \left\{ \sum_{k=1}^{l} p_k(T)x_k \, \, \bigg|\, \, p_1, \dots, p_k \text{ are {complex polynomials}} \right\}}
\end{align*}
coincides with $X$ and $l$ is {minimal in terms of this property}, cf.~for instance 
\cite{herrero1985multicyclic}. 
That is, $(A,B)$ is ensemble reachable if and only if $\operatorname{im} \mathcal{M}_B$ is a cyclic subspace of
$\operatorname{im} \mathcal{M}_A$, cf. \cite{Herrero_McDonald_1983,Nikolskii1982}. Moreover, if $(A,B)$ is ensemble
reachable then the matrix multiplication operator ${\cal M}_A$ is $l$-multicyclic with $l \leq m$.

%
{\begin{remark}
    Another quite general characterizations for approximate reachability/controllability is given by \cite[Theorem~4.1.7~(b)]{curtainzwart95}
    which is however often hard to check. In special cases, where additional spectral information on the operator ${\cal A}$ is available,
    more explicit condition can be derived. A standard assumption of this type is the existence of a Riesz basis of eigenvectors of the
    operator ${\cal A}$, cf.~\cite[Section~4.2]{curtainzwart95} and \cite{Jacob2006diagonal}. But except for some trivial cases where,
    e.g., $A(\theta)$ has constant eigenvalues, the multiplication operator ${\cal M}_A$ induced by $A(\theta)$ does not have a point
    spectrum and therefore these results are in general not applicable.\\  
    In \cite{loheac:hal-01164525}, ensemble reachability ($=$ approximate simultaneous controllability) for continuous-time
      systems is linked to an averaged reachablity problem\footnote{A precise definition of {\it averaged reachability} is given in Section
      \ref{sec:output}.} via a constrained optimal control problem including a $L^2$-penalty term. The authors show that by increasing
    the penalty parameter and solving the corresponding optimal control problem one gains a sequence of controls which drive the system
    (approximately) to the desired target state.
  However, the result does not provide any explicit conditions for ensemble reachability in terms of the matrix families $A(\theta)$
  and $B(\theta)$.
  \end{remark}}

We start off our further investigations with two auxiliary results. The first is devoted to the fact that
previous works on ensemble reachability were often limited to pairs of real matrix families $(A(\theta),B(\theta))$.
The second characterizes how ensemble reachability behaves under restrictions. {Let $X^{\R}(\p)$
be a separable M-Banach space of real-valued functions. Then $X^{\C}(\p):=\{ g+ih\, |\, g,h \in X^{\R}(\p)\}$
denotes its complexification equipped with the norm
$\|g+ih\|_{{X^{\C}(\p)}} := \max_{t \in [0,2\pi]} \|g\cos(t) - h\sin(t)\|_{{X^{\R}(\p)}}$.}
For details we refer to \cite{munoz_studia_1999}.
\begin{lemma}\label{lem:complexification}
Let $(A,B)$ be a real, i.e. $(A,B) \in C^{\R}_{n,n}(\p) \times X^{\R}_{n,m}(\p)$. Then $(A,B)$ is ensemble
reachable on $X^{\C}_n(\p)$ if and only if $(A,B)$ is ensemble reachable on $X^{\R}_n(\p)$,
{i.e. if for all $f \in X_n(\p)$ and $\e >0$ there {exist} $T \geq 0$ and a $\mathbb{R}^m$-valued
input $u \in {U(T)}$ such that
\begin{align}\label{eq:openloop}
\|\varphi(T,u,0) -f\|_{X_n(\p)} < \e.
\end{align}}
\end{lemma}
\begin{proof}
  {Obviously, condition \eqref{eq:density_cond_er} applies to the real case as well and yields
    the following equivalent characterization of ensemble reachability on $X^{\R}_n(\p)$: For each $\e>0$ and for
    each $f\in X^{\R}_n(\p)$ there are real polynomials $p_1,...,p_m$  such that
\begin{align*}
   \Big\| \sum_{j=1}^{m}p_j(A)\,b_j - f \Big\|_{X^{\R}_n(\p)} < \e.
\end{align*}}
{For simplicity, let $m=1$ in the sequel. First, suppose that $(A,B)$ is ensemble reachable on
  $X^{\C}_n(\p)$ and let $\e>0$ and $f\in X^{\R}_n(\p)$.} Then, there is a complex polynomial $p(z)=c_0+c_1z+\cdots +c_k \,z^k$ such that 
\begin{align*}
\| p(A)\,b - f \|_{X^{\C}_n(\p)} < \e 
\end{align*}
In particular, for
$r(z) = \operatorname{Re}(c_0)+\operatorname{Re}(c_1) \, z +\cdots +\operatorname{Re}(c_k) \,z^k$ and
$q(z) = \operatorname{Im}(c_0)+\operatorname{Im}(c_1) \, z+\cdots +\operatorname{Im}(c_k) \,z^k$ it follows
from \cite[Proposition~1]{munoz_studia_1999}
\begin{align*}
\| r(A)\,b - f \|_{X^{\R}_n(\p)}  \leq  \| r(A)\,b - f  + i q(A)b\|_{X^{\C}_n(\p)}
= \| p(A)\,b - f \|_{X^{\C}_n(\p)}<\e\,. 
\end{align*}

Conversely, let $\e>0$ and $f=g+ih \in X^{\C}_n(\p)$. By assumption, there are real polynomials $r$ and $q$ such that
\begin{align*}
 \| r(A)\,b - g \|_{X^{\R}_n(\p)} < \tfrac{\e}{2} \quad \text{ and } \quad \| q(A)\,b - h \|_{X^{\R}_n(\p)} < \tfrac{\e}{2}.
\end{align*}
Thus, defining $p(z):= r(z)+iq(z)$ we have
\begin{align*}
 \| p(A)\,b - f \|_{X^{\C}_n(\p)} < \e.
\end{align*}
This shows the assertion.
\end{proof}

\medskip

Next, as mentioned before, we {treat} restrictions of parameter-dependent systems to subsets of
the parameter space.
For $\p_1 \subset \p_2$ we say that a pair $\big(X(\p_2),X(\p_1)\big)$ has the {\it restriction property} if the
restriction operator ${\cal R} \colon X(\p_2) \to X(\p_1)$, ${f \mapsto {\cal R}f := f\vert_{\p_1}}$
is well-defined, bounded and onto. This yields the following result.
{
\begin{lemma}\label{lem:resriction-property}
  Let $\p_1, \p_2$ be compact with $\p_1 \subset \p_2$. If $\big(X(\p_1),X(\p_2)\big)$ has the
  restriction property and $(A,B) \in C_{n,n}(\p_2) \times X_{n,m}(\p_2)$ is ensemble reachable on $X_n(\p_2)$
  then $({\cal R}A, {\cal R}B) \in C_{n,n}(\p_1) \times X_{n,m}(\p_1)$ is ensemble reachable on $X_n(\p_1)$.
\end{lemma}}

\begin{proof}
Let $\e>0$ and $f \in X_n(\p_1)$.   By assumption, the restriction operator ${\cal R}$
is onto, i.e.~there exists $\tilde f \in X_n(\p_2)$ with ${\cal R}\tilde f = f$. Moreover, ensemble 
reachability of $(A,B)$ on $X_n(\p_2)$ implies the existence of polynomials $\tilde p_j$ such that
$\Big\| \sum_{j=1}^m \tilde p_j(A)b_j - \tilde f \Big\|_{X_{n}(\p_2)} < \e$ holds. Hence, one has the 
estimate
\begin{align*}
\Big\| \sum_{j=1}^m \tilde p_j({\cal R} A){\cal R} b_j - f \Big\|_{X_{n}(\p_1)}
=\Big\| {\cal R}\Big(\sum_{j=1}^m \tilde p_j(A)b_j - \tilde f\Big) \Big\|_{X_{n}(\p_1)}
\leq \eta \Big\|  \sum_{j=1}^m \tilde p_j(A)b_j - \tilde f \Big\|_{X_{n}(\p_2)} < \eta \e,
\end{align*}
where $\eta$ denotes the operator norm of $\mathcal{R}$, and thus ensemble reachability of $({\cal R}A , {\cal R}B)$ follows.
\end{proof}

{
  \begin{remark}\label{rem:restriction-property}
    Two standard scenarios where the restriction property is satisfied are $C$- and $L^q$-spaces. More precisely,
    one has:
\begin{enumerate}
\item[(a)] If $\p_2$ is compact and $\p_1 \subset \p_2$ {is} closed, then Tietze's extensions theorem
  \cite[{Theorem}~20.4]{rudin1987realcomplex} implies that the pair $\big(C(\p_2),C(\p_1)\big)$
  has the restriction property.
\item[(b)]
  If $(\p_2,\mu)$ is a measure space and $\p_1 \subset \p_2$ a measurable subset then the pairs
  $\big(L^q(\p_2),L^q(\p_1)\big)$ have obviously the restriction property for $1 \leq q \leq \infty$. 
\end{enumerate}
\end{remark}}

\medskip

\subsection{Cascade structures}
\label{subsec:cascade-structures}

{Here, we consider probably the most simplest case of an interconnected system{:}  a
 {\emph{cascade}.} Let $(A_{ij},B_{ij})\in C_{n_i,n_j}(\textbf{P})\times X_{n_i,m_j} (\p)$
with $1\leq i\leq j\leq N$ a finite collection of linear parameter-dependent systems and set 
$\overline{n}:={n_1+ \cdots + n_N}$ and $\overline{m}:={m_1+\cdots +m_N}$. Moreover, define the associated
upper triangular parameter-dependent pairs by}
\begin{equation}\label{eq:uppertriang}
 A := 
\begin{pmatrix}
 A_{11}  & \cdots &A_{1N}\\
 & \ddots &\vdots\\
 0& & A_{NN} 
\end{pmatrix}
\in C_{\overline{n},\overline{n}}  (\mathbf{P})
,\quad 
B : = 
\begin{pmatrix}
 B_{11}  & \cdots &B_{1N}\\
 & \ddots &\vdots\\
 0& & B_{NN} 
\end{pmatrix}
\in X_{\overline{n},\overline{m}}(\mathbf{P}).
\end{equation}

This upper triangular structure {guarantees} a straightforward sufficient reachability condition.
We note that this result is not limited to ensembles, indeed it is a consequence of the cascade structure
(cf.~\cite{sontag}) and holds for any linear system. {Since the result will be used later in relation
to Theorem~\ref{thm:CC-MI-sufficient-conditions}  as well as Propositions~\ref{prop:A_B_triangular} and~\ref{prop:Lq-MI-sufficient-conditions} we provide a sketch of its straightforward proof}.

\begin{proposition}\label{prop:technical_MATCOM}
If the diagonal pairs $(A_{ii}, B_{ii})$ of the upper triangular pair $(A,B)$ given by \eqref{eq:uppertriang}
are ensemble reachable on $X_{n_i}(\p)$ for all $i=1, \ldots,N$ then
 {$(A,B)$} is ensemble reachable on $X_{n_1}(\p) \times \cdots \times X_{n_N}(\p)$.
\end{proposition}
\begin{proof}
  Suppose the diagonal pairs $(A_{ii},B_{ii})$ are ensemble reachable on $X_{n_i}(\p)$ for all $i=1,...,N$. We start
  with the discrete-time {case and assume $N=2$ for simplicity. This fully captures the key idea.
    The general case can be treated by induction.} Thus, we consider the discrete-time system
\begin{equation}\label{eq:proof-discrete_upper-triangular}
 \begin{split}
{x_1}({k}+1)&=A_{11}x_1({k})+A_{12}x_2({k}) +B_{11}u_1({k})+B_{12}u_2({k})\\
{x}_2({k}+1)&=A_{22}x_2({k}) +B_{22}u_2({k}).
\end{split}
\end{equation}
Let $\psi_1(t,u,0)$ and $\psi_2(t,u,0)$ denote the solution of the systems defined by $(A_{11},B_{11})$ and
$(A_{22},B_{22})$, respectively, i.e.
\begin{align*}
\psi_1(t,u,0) = \sum_{k=0}^{t-1} A_{11}^kB_{11}u(t-1-k) \quad \text{{and}} \quad \psi_2(t,u,0) =\sum_{k=0}^{t-1} A_{22}^kB_{22}u(t-1-k).
\end{align*}

Let $f=\binom{f_1}{f_2} \in X_{n_1+n_2}(\mathbf{P})$ and let $\varepsilon>0$.  Then, ensemble reachability of
$(A_{22},B_{22})$ implies the existence of a time $T_2>0$  and an input $v_2\in {\C^{m_2}\times \cdots \times \C^{m_2} }$
such that 
\[
\| \psi_2(T_2,v_2,0) - f_2 \|_{X_{n_2}(\mathbf{P})} < \varepsilon.
\]
Let
\begin{align*}
\tilde f_2 := \sum_{k=0}^{T_2-1} A_{11}^k \Big(  A_{12} \psi_2(T_2-1-k,v_2,0) +   B_{12}v_2(T_2-1-k)\Big).
\end{align*}
Then, since $(A_{11},B_{11})$ is ensemble reachable, for $\tilde f_1 :=f_1-\tilde f_2$ and $\e>0$ there is a time $T_1>0$
and an input $v_1 \in {\C^{m_1}\times \cdots \times \C^{m_1} } $ such that 
\[
\| \psi_1(T_1,v_1,0) - \tilde f_1 \|_{X_{n_1}(\mathbf{P})} < \varepsilon.
\]
Finally, without loss of generality we {may assume $T_1 = T_2=:T$. Thereby, we can show that
at time $T \in \N$ and for the input $u=\binom{v_1}{v_2}$ the solution $\varphi = \binom{\varphi_1}{\varphi_2}$
of \eqref{eq:proof-discrete_upper-triangular} satisfies}
\[
\|\varphi(T,u,0)- f  \|_{X_{n_1}(\mathbf{P})\times X_{n_2}(\mathbf{P})} < \varepsilon.
\]
Clearly, one has $\varphi_2(T,u,0)= \psi_2(T,v_2,0)$. Also, note that for $t \leq T$ it holds
\begin{align*}
\varphi_1(t,u,0) = \psi_1(t,v_1,0) + \sum_{k=0}^{t-1} A_{11}^k \big( A_{12} \psi_2(t-1-k,v_2,0) +   B_{12}v_2(t-1-k)\big)
\end{align*}
and therefore
\begin{align*}
 \varphi_1(T,u,0) -f_1 =  \psi_1(T,v_1,0) + \tilde f_2 - f_1  = \psi_1(T,v_1,0) - \tilde f_1.
\end{align*}
This shows the assertion in the discrete-time case. {The continuous-time case follows mutatis mutandis.}
\end{proof}

\medskip
\noindent
The latter statement is an extension and correction of \cite[Proposition~2]{schonlein2016controllability}, 
where it was claimed that the reverse implication also holds. This, however, is false in general, see 
Remark~\ref{rem:upper-triangular-not-nec}.

\subsection{Parallel structures}
\label{subsec:parallel-structures}

Next we want to analyse parallel connections of linear ensembles. Again, we want to
emphasise that our result (Theorem \ref{thm:parallel_multi}) applies in principle to arbitrary linear
systems on Banach spaces even though verifying the respective spectral conditions is in general very
difficult. For linear ensembles, however, the spectral conditions often allow a simplified test.
More on this issue can be found  {at the beginning of Subsection \ref{subsec:decomposition-techniques}}.

Before stating and proving our result we recall that a set $\Omega \subset \C$ is non-separating
  if $\C\setminus \Omega$ is connected and introduce the notation
  \begin{equation*}
    \sigma_{X_{n}(\p)} \big(\mathcal{M}_{A}\big) \subset \C\,,
  \end{equation*}
for the spectrum of the multiplication operator $\mathcal{M}_{A} : X_{n}(\p) \to X_{n}(\p)${.}

\begin{theorem}\label{thm:parallel_multi}
Let $\textbf{P}\subset \mathbb{C}$ be compact. Suppose the pairs $(A_i,B_i) \in C_{n_i,n_i}(\p) \times X_{n_i,m}(\p)$,
$i = 1, \dots, N$ satisfy the following conditions:
\begin{enumerate}
\item[(a)]
$(A_i,B_i)$ are ensemble reachable on $X_{n_i}(\p)$ for $i = 1, \dots, N$.
\item[(b)]
$\sigma_{X_{n_i}(\p)} \big(\mathcal{M}_{A_i}\big)$ has only finitely many connected components for $i = 1, \dots, N$.
\item[(c)]
$\sigma_{X_{n_i}(\p)} \big(\mathcal{M}_{A_i}\big)$ is non-separating for $i = 1, \dots, N$.
\item[(d)]
$\sigma_{X_{n_i}(\p)} \big(\mathcal{M}_{A_i}\big) \cap \sigma_{X_{n_j}(\p)} \big(\mathcal{M}_{A_j}\big) = \emptyset$
for $i\neq j$. 
\end{enumerate}
Then, the parallel connection given by the pair
 \begin{align*}
\left(\begin{pmatrix} A_1& & \\ & \ddots & \\ & & A_N \end{pmatrix} , 
\begin{pmatrix} B_1 \\ \vdots\\ B_N \end{pmatrix}\right)  \in C_{\bar n, \bar n} (\p)\times X_{\bar n ,m}(\mathbf{P}),
\quad \bar n = n_1+\cdots+ n_N
 \end{align*}
is ensemble reachable on $X_{n_1}(\p) \times \cdots \times X_{n_N}(\p)$.
\end{theorem}
\begin{proof}
The proof will be given for the case $N=2$, $m=1$, i.e.~for two single-input pairs 
$(A_1,b_1) \in C_{n_1,n_1}(\p) \times X_{n_1}(\p)$ and $(A_2,b_2) \in C_{n_2,n_2}(\p) \times X_{n_2}(\p)$. The 
arguments easily extend to the general case. Let $\e>0$ and $f=\binom{f_1}{f_2}\in X_{n_1}(\p) \times X_{n_2}(\p)$
be fixed. Since the pairs $(A_1,b_1)$ and $(A_2,b_2)$ are ensemble reachable there are  polynomials $p_1$ and
$p_2$ such that
\begin{align*}
 \| p_1(A_1)b_1 - f_1\|_{X_{n_1}(\p)} < \tfrac{\e}{2}   \qquad \text{ and } \qquad \| p_2(A_2)b_2 - f_2\|_{X_{n_2}(\p)} < \tfrac{\e}{2}.
\end{align*}
By assumption the compact sets {$\sigma_{X_{n_1}(\p)} \big(\mathcal{M}_{A_1}\big)$ and
  $\sigma_{X_{n_2}(\p)} \big(\mathcal{M}_{A_2}\big)$} are disjoint, do
not separate the plane and have only finitely many pairwise disjoint connected components. Thus, an application
of Lemma~\ref{lem:extension}~(see Appendix~\ref{appendix:1}) yields disjoint compact sets $K_1$ and $K_2$ which
do not separate the plane and properly contain {$\sigma_{X_{n_1}(\p)} \big(\mathcal{M}_{A_1}\big)$
and $\sigma_{X_{n_2}(\p)} \big(\mathcal{M}_{A_2}\big)$}, respectively. We note that, the polynomials $p_1,p_2$
and the {compact sets} $K_1$, $K_2$ are fixed {from now on}. Next, we consider
the functions
\begin{align*}
 h_1 \colon K_1 \cup K_2 \to \C, \quad h_1(z)= 
 \begin{cases}
  1 & \text{ if } z \in K_1\\
  0 & \text{ if } z \in K_2
 \end{cases}\\
  h_2 \colon K_1 \cup K_2 \to \C, \quad h_2(z)= 
 \begin{cases}
  0 & \text{ if } z \in K_1\\
  1 & \text{ if } z \in K_2.
 \end{cases}
\end{align*}
Then, for every $\tilde \e >0$, Lemma~\ref{lem:runge-extension}~(see Appendix~\ref{appendix:1}) implies the
existence of polynomials $q_1$ and $q_2$ such that
\begin{align*}
 | h_1(z) - q_1(z)| < \tilde\e \quad \text{ and } \quad | h_2(z) - q_2(z)| < \tilde\e
\quad \text{ for all } z \in K_1 \cup K_2.
\end{align*}
{In particular, one has}
\begin{align*}
 |q_1(z)| <  (1+\tilde\e) \quad \text{ and } \quad  |q_2(z)| < \tilde\e   \qquad \text{ for all } z \in K_1.
\end{align*}
Defining the polynomial 
\begin{align*}
 p(z) := q_1(z)\, p_1(z) + q_2(z)\, p_2(z)
\end{align*}
we shall prove
\begin{align*}
\left\| 
\begin{pmatrix} p(A_1)b_1 - f_1 \\ p(A_2)b_2 - f_2 \end{pmatrix}
 \right\|_{X_{n_1}(\p) \times X_{n_2}(\p)} \leq \e.
\end{align*}
Without loss of generality we consider only the first component and show
\begin{align*}
\| p(A_1)b_1 - f_1 \|_{X_{n_1}(\p)} \leq \e.
\end{align*}
Obviously, one has the estimate
\begin{align*}
\| & p(A_1)b_1 - f_1 \|_{X_{n_1}(\p)} =
\| q_1(A_1)\, p_1(A_1)b_1 + q_2(A_1)\, p_2(A_1)b_1 - f_1 \|_{X_{n_1}(\p)} \\
& \leq \| q_1(A_1)\, p_1(A_1)b_1 - p_1(A_1)b_1\|_{X_{n_1}(\p)} + \| p_1(A_1)b_1 - f_1\|_{X_{n_1}(\p)} 
+ \|q_2(A_1)\, p_2(A_1)b_1 \|_{X_{n_1}(\p)} \\
   & {\leq \| (q_1(A_1)-I) \, p_1(A_1)\| \, \|b_1\|_{X_{n_1}(\p)}
     + \|p_1(A_1)b_1 - f_1 \|_{X_{n_1}(\p)} + \| q_2(A_1) \, p_2(A_1)\| \, \|b_1 \|_{X_{n_1}(\p)},}
\end{align*}
where {$\| (q_1(A_1)-I) \, p_1(A_1)\|$ and $ \| q_2(A_1) \, p_2(A_1)\|$} denote
the respective operator norms 
on $X_{n_1}(\p)$. Using the Dunford-Taylor formula, cf. \cite[Chapter~1, \S~5, Section~6]{kato}, for any
polynomial $q$ one has
\begin{align*}
 q(A_1(\theta))= \frac{1}{2 \pi i } \int_{\gamma} q(z) (zI-A_1(\theta))^{-1} \di{z},
\end{align*} 
where $\gamma$ consists of finitely many positive oriented simple closed grid polygons in
$K_1 \setminus \sigma_{X_{n_1}(\p)} \big(\mathcal{M}_{A_1}\big)$ and its trace is denoted by $\operatorname{tr} \gamma$.
Note that the polynomial $\tilde q_1(z):=q_1(z)-1$ satisfies $|\tilde q_1(z)|<\tilde\e$ for all $z \in K_1$ and thus we have 
\begin{align*}
\| (q_1(A_1)-I) \, p_1(A_1)\|
& 
  \leq \frac{1}{2 \pi} \int_{\gamma} |\tilde q_1(z)| \, |p_1(z)| \, \| (zI-\mathcal{M}_{A_1})^{-1}\| \di{z} \\ 
  & \leq \frac{\tilde\e \, L_{\gamma}}{2 \pi} \,\max_{z \in \operatorname{tr} \gamma}  |p_1(z)| \, \|(zI-\mathcal{M}_{A_1})^{-1}\|\,.
\end{align*}
Similarly, it follows
\begin{align*}
  \| q_2(A_1) \, p_2(A_1)\|
  & \leq \frac{\tilde\e \, L_{\gamma}}{2 \pi} \,\max_{z \in \operatorname{tr} \gamma}  |p_2(z)| \, \|(zI-\mathcal{M}_{A_1})^{-1}\|\,.
\end{align*}
Then, setting 
$\alpha_1:=\max_{z \in \operatorname{tr} \gamma} |p_1(z)|\|(zI-\mathcal{M}_{A_1})^{-1}\|$, 
$\alpha_2:= \max_{z \in \operatorname{tr} \gamma} |p_2(z)|\|(zI-\mathcal{M}_{A_1})^{-1}\|$,
and $\beta_1:=\|b_1\|_{X_{n_1}(\p)}$, 
%
we obtain
\begin{align*}
\| p(A_1)b_1 - f_1 \|_{X_{n_1}(\p)}
\leq  \tfrac{\tilde \e \, L_{\gamma}}{2\pi} \, \alpha_1 \, \beta_1  
+  \tfrac{\e}{2} +  \tfrac{\tilde \e \, L_{\gamma}}{2\pi} \, \alpha_2 \, \,\beta_1.
\end{align*}
Thus the claim follows by picking 
$\tilde \e < \tfrac{\e \pi}{L_{\gamma} \, \beta_1 (\alpha_1 + \alpha_2)}$.
\end{proof}

\medskip
\noindent
{A straightforward analysis of the above proof reveals that one can also allow $\p$
to depend on $i$. For simplicity of notation, we state the corresponding result only for $N=2$.}
{
\begin{corollary}
\label{cor:parallel_multi}
Let $\textbf{P}_1, \textbf{P}_2\subset \mathbb{C}$ be compact. Suppose the pairs
$(A_1,B_1) \in C_{n_1,n_1}(\p_1) \times X_{n_1,m}(\p_1)$ and $(A_2,B_2) \in C_{n_2,n_2}(\p_2) \times Y_{n_2,m}(\p_2)$
satisfy the following conditions:
\begin{enumerate}
\item[(a)]
$(A_1,B_1)$ and $(A_2,B_2)$ are ensemble reachable on $X_{n_1}(\p_1)$ and $Y_{n_2}(\p_2)$, respectively.
\item[(b)]
  $\sigma_{X_{n_1}(\p_1)} \big(\mathcal{M}_{A_1}\big)$ and $\sigma_{Y_{n_2}(\p_2)} \big(\mathcal{M}_{A_2}\big)$
  have only finitely many connected components.
\item[(c)]
  $\sigma_{X_{n_1}(\p_1)} \big(\mathcal{M}_{A_1}\big)$ and $\sigma_{Y_{n_2}(\p_2)} \big(\mathcal{M}_{A_2}\big)$ are non-separating.
\item[(d)]
$\sigma_{X_{n_1}(\p_1)} \big(\mathcal{M}_{A_1}\big) \cap \sigma_{X_{n_2}(\p_2)} \big(\mathcal{M}_{A_2}\big) = \emptyset$. 
\end{enumerate}
Then, the parallel connection given by the pair
$ \left(\begin{pmatrix} A_1& 0 \\ 0 & A_2 \end{pmatrix} , 
\begin{pmatrix} B_1 \\ B_2 \end{pmatrix}\right)$
is ensemble reachable on $X_{n_1}(\p_1) \times  Y_{n_2}(\p_2)$.
\end{corollary}}

\begin{remark}\label{rem:thm1}
  \begin{enumerate}
  \item[(a)]
We note that the technique of the construction of the polynomial $p =q_1p_1 + q_2 p_2 $ is well-known in complex
approximation, cf. e.g. \cite{andrievskii2005polynomial}. This construction can also be extended such that
interpolation properties of the polynomials $p_1$ and $p_2$ are transferred  to the polynomial $p$. 
\item[(b)] {In the finite-dimensional case, a {complete} characterization for a parallel connection of reachable
    systems to be reachable is due to Fuhrmann in 1975. For a comprehensive exposition on this topic we refer to
    the textbook \cite[Section~1.1]{fuhrmann2015mathematics} (and the references therein).}
\end{enumerate}
\end{remark}

\medskip
The following example illustrates the necessity of the spectral conditions of Theorem~\ref{thm:parallel_multi}
and furthermore demonstrates that the spectrum of a multiplication operator $\mathcal{M}_A$ is not necessarily
given by the union of all pointwise spectra $\sigma(A(\theta))$.

\begin{example}\label{example:non-sep}
{
Consider $\p := \partial \mathbb{D}$ and
\begin{equation}\label{eq:sys_ex_2}
A(\theta) := \begin{pmatrix} \theta & 0 \\ 0 &  \frac{1}{2}\theta \end{pmatrix}\,,
\quad b(\theta) := \begin{pmatrix}1 \\ 1\end{pmatrix}
\quad\text{for}\quad \theta \in \p\,. 
\end{equation}
Let $\mathbb{D}$ denote the unit disc and let ${\cal A}(\mathbb{D})$ denote the disc algebra, that is
$f\in {\cal A}(\mathbb{D})$ if and only if $f\colon \mathbb{D}\to \C$ is holomorphic and $f$ extends
continuously to $\overline{\mathbb{D}}$. Moreover, let $X_2(\p) := X(\p) \times X(\p)$, where $X(\p)$
is the space of boundary values of disc algebra functions, i.e.~$X(\p)$ is the subspace of $C(\p)$
which can be associated with disc algebra functions in the sense 
\begin{equation}\label{eq:disc-algebra}
{X(\p)} := \{f \in C(\p)\;|\; \exists \, g \in {\cal A}(\mathbb{D}) \, : \, f= g\big|_{\p}\}\,.
\end{equation}
Then, the subsystems $(A_1,b_1)$ and $(A_2,b_2)$ are ensemble reachable on $X(\p)$. However, $(A,b)$ is not ensemble
reachable on {$X_2(\p)$}. This can be seen as follows: Let $f=\binom{f_1}{f_2} \in X_{{2}}(\p)$ be given.
To prove ensemble reachability of $(A_1,b_1)$ we exploit the fact that $\cal A(\mathbb{D})$ coincides with the
closure all complex polynomials with respect to the maximum norm, cf. \cite[Sec.~6.2]{zhu_book}. Hence, for
$f_1 \in X(\p)$ and $g \in {\cal A}(\mathbb{D})$ such that \eqref{eq:disc-algebra} holds we can find a sequence
$p_n$ of complex polynomials such that 
\begin{equation*}
\max_{z \in \overline{\mathbb{D}}}\| p_n(z) - g(z)\| \to 0 \quad \text{\rm as} \; n \to \infty\
\end{equation*}
and, in particular
\begin{equation}\label{eq:approximation-disc-algebra}
\max_{\theta \in \p}\| p_n(\theta) - f_1(\theta)\| = \max_{\theta \in \p}\| p_n(\theta) - g(\theta)\| 
\to 0 \quad \text{\rm for} \; n \to \infty\,. 
\end{equation}
This show that $(A_1,b_1)$ is ensemble reachable. Moreover, the identities $A_2 = \frac{1}{2} A_1$ and $b_2 = b_1$
immediately imply that $(A_2,b_2)$ is also ensemble reachable. It follows, however, by the maximum principle, that
any  sequence $p_n$ which satisfies \eqref{eq:approximation-disc-algebra} has to converges uniformly
on $\overline{\mathbb{D}}$ to $g \in {\cal A}(\mathbb{D})$. In particular, one has that $p_n(\frac{\theta}{2})$
converges uniformly to $g(\frac{\theta}{2})$ for all $\theta \in \p$, i.e. there is no degree of freedom for
choosing $f_2$. Hence $(A,b)$ is not ensemble reachable.}
\end{example}
\medskip

Now the question arises why ensemble {reachability} fails in the above example. First, one might
think that {the non-separating condition~(c)} is violated, because the union of the
pointwise spectra $\sigma(A_1(\theta))$ yields the unit circle $\partial \mathbb{D}$ which is obviously
a separating subset of $\C$. But a more thorough analysis shows that condition (d), the disjointness
condition fails, because the spectra $\sigma_{X(\p)}(\mathcal{M}_{A_1})$ and $\sigma_{X(\p)}(\mathcal{M}_{A_2})$
{coincide with $\overline{\mathbb{D}}$ and $\frac{1}{2}\overline{\mathbb{D}}$, respectively.
This follows straightforwardly} from the maximum principle and the fact that the spectrum of the multiplication
operator $f(z) \mapsto zf(z)$ on the disc algebra ${\cal A}(\mathbb{D})$ is given by the unit disc
$\overline{\mathbb{D}}$.

\subsection{Decomposition techniques}
\label{subsec:decomposition-techniques}

\medskip
\noindent
{By the above example we have seen that the spectrum of ${\cal M}_A: X_n(\p) \to X_n(\p)$ does
  not always coincide with the union
  \begin{align}\label{eq:spectrum}
    \bigcup_{\theta \in \mathbf P} \sigma \left( A(\theta) \right) =: \operatorname{spec} A (\p)
  \end{align}
  of all {pointwise spectra} $\sigma \left( A(\theta) \right)$. More precisely, it demonstrates that the inclusion
  \begin{align}\label{eq:spectrum2}
    \sigma_{X_n(\p)}(\mathcal{M}_{A}) \subset \bigcup_{\theta \in \mathbf P} \sigma \left( A(\theta) \right)
  \end{align}
  is in general false. However, if $X_n(\p)$ is a M-space  $\mathcal{M}_{(A - \lambda I)^{-1}}$ is obviously
  the (bounded) inverse of $\mathcal{M}_{A - \lambda I}$ whenever $\lambda \in \C \setminus \operatorname{spec} A (\p)$
  and thus \eqref{eq:spectrum2} holds for every M-space while equality cannot be guaranteed as the following example
  illustrates: $\p := [-1,1]$ and
  $X([-1,1]) := \big\{x \in C([-1,1]) \;|\; x(\theta) = 0 \;\text{for all}\; \theta \in [-1,0]\big\}$ and
  $a(\theta) = \theta$. Yet, in many standard case, like $X_n(\p) = C_n(\p)$ or $X_n(\p) = L_n^q(\p) $, one has
  equality in \eqref{eq:spectrum2}, cf.~\cite{holderrrieth1991}.}

{Now a promising strategy for M-spaces is based on the idea of decomposing $\operatorname{spec} A (\p)$
  such that the matrix pair $(A,B)$ can be continuously transformed into a simplifying block structure as above in
  Theorem \ref{thm:parallel_multi}. To this end, we introduce the set-valued {\it spectral map}
$\operatorname{spec} A \colon \p \rightsquigarrow \C$, 
\begin{align*}
\operatorname{spec} A(\theta) :=  \sigma \left( A(\theta) \right)\,.
\end{align*}
Thereby, \eqref{eq:spectrum2} is equivalent to say that the image of the spectral map 
contains the spectrum of the matrix-multiplication operator ${\cal M}_A$.}
{A set-valued map $\Gamma: \p \rightsquigarrow \C$ is termed {\it partial spectral map} if the
  inclusion $\Gamma(\theta) \subset \sigma(A(\theta))$ is satisfied for all $\theta \in \p$ and it is called
  {\it continuous} if continuity with respect to the Hausdorff metric holds. {Moreover,
    a single-valued partial spectral map}
  is referred to as {an}
{\it eigenvalue selection} and will be denoted by $\lambda: \p \to \C$. Two partial spectral maps $\Gamma_1$
and $\Gamma_2$ are {\it pointwise disjoint} if $\Gamma_1(\theta) \cap \Gamma_2(\theta) = \emptyset$ for all
$\theta \in \p$. They are {\it strictly disjoint} if {one has} $\Gamma_1(\p) \cap \Gamma_2(\p) = \emptyset$.
Obviously, strict disjointness implies pointwise disjointness. Finitely many (continuous) partial spectral
maps $\Gamma_1, \dots, \Gamma_N$ are called a (continuous) {\it spectral decomposition} of $A$ if
\begin{equation*}
\bigcup_{i = 1}^N  \Gamma_i(\theta) = \sigma(A(\theta))
\end{equation*}
for all $\theta \in \p$. If $\Gamma_1, \dots, \Gamma_N$ are additionally single-valued the spectral decomposition
will be call{ed} \emph{single-valued}. Note, that $\Gamma_1, \dots, \Gamma_N$ are not required
to be disjoint in any sense.
However, if $\Gamma_1, \dots, \Gamma_k$ are pairwise pointwise/strictly disjoint the spectral decomposition
will be call \emph{pointwise/strictly disjoint}.} Certainly, there exists always a continuous spectral decomposition
of $A$, for instance the trivial one $\Gamma(\theta) := \sigma(A(\theta))$, and sometimes this is even the only one
which is continuous as in the case
\begin{align*}
A(\theta) := \begin{pmatrix} 0 & 1\\ \theta & 0 \end{pmatrix}\,,
\quad \theta \in \p := \overline{\mathbb{D}}\,.
\end{align*}

\noindent
However, locally or if $\p$ has nice topological properties one can obtain continuous spectral decompositions
which are considerably finer.
\begin{lemma}\label{lem1:spectral-family}
Let $\p \subset \C$ be compact and $A \in C_{n,n}(\p)$.
\begin{enumerate}
\item[(a)]
  For every {relatively} open subset $U \subset \p$ there exists an {relatively} open subset $V \subset U$ such that the restriction $A|_V$
  allows a continuous {single-valued spectral decomposition}.
\item[(b)]
  If $\p$ is homeomorphic to $[0,1]$ then there exists a global continuous {single-valued spectral decomposition}
  for $A$.
\item[(c)]
  If $\p$ {is} contractible, locally path-connected and if the eigenvalues of $A(\theta)$ are simple for all
  $\theta \in \p$ then there exists a global continuous single-valued spectral decomposition.
\end{enumerate}
\end{lemma}

%
%

\medskip
\noindent
The proof of Lemma~\ref{lem1:spectral-family} is given in Appendix~\ref{appendix:2}. The arguments verifying
part~(c) actually indicate that even continuous eigenvector selections should be possible in the contractible
case. Our next result shows that this is in fact true {even without simplicity assumption on the eigenvalues
once a suitable spectral decomposition can be guaranteed. {The statement should be known to experts. However,
as the standard literature on perturbations theory focuses on the analytic case}, cf. \cite{baumgaertel,kato},
we could not locate an appropriate reference. {Hence we will provide a proof in Appendix~\ref{appendix:3}}.


\begin{proposition}\label{prop:spectral-family}
Let $\p \subset \C$ be compact and contractible and let $A \in C_{n,n}(\p)$. Assume that $\Gamma_1, \dots, \Gamma_N$
is a pointwise disjoint continuous spectral decomposition. Then there exists a continuous family of invertible
matrices $T(\theta)$ such that 
\begin{equation}\label{eq:block-diagonal}
T(\theta)^{-1} A(\theta) T(\theta) = 
\begin{pmatrix}
A_1(\theta) & & 0\\
 & \ddots & \\
0 & & A_N(\theta)
\end{pmatrix}
\end{equation}
and the spectra of $A_i(\theta)$ are given by $\Gamma_i(\theta)$ for all $\theta \in \p$ and $i = 1, \dots, N$.
\end{proposition}

{In the sequel, assume that $\Gamma_1, \dots, \Gamma_N$ is a pointwise disjoint continuous
spectral decomposition of $(A,B) \in C_{n,n}(\mathbf{P}) \times X_{n,m}(\p)$.} Then the subsystems
$(A_i,B_i) \in C_{n_i,n_i}(\mathbf{P}) \times X_{n_i,m}(\p)$ given by 
\begin{equation}
\label{eq:subsystems}
A_i(\theta) := \Pi_i T^{-1}(\theta)A(\theta)T(\theta) \Pi_i^{*}\,,\quad\quad
B_i(\theta) := \Pi_i T^{-1}(\theta)B(\theta)
\end{equation}
{with} $\Pi_i := \big(0 \dots 0 \;I_{n_i}\; 0 \dots 0 \big) \in \mathbb{C}^{n_i \times n}$ which result
from Proposition~\ref{prop:spectral-family} are called the {\it associated subpairs}. Moreover, the multiplication
operator $\mathcal{M}_T: X_n(\p) \to X_n(\p)$, $\mathcal{M}_Tx(\theta) := T(\theta)x(\theta)$ is {termed
{\it associated transformation map}. Note that the subsystems $(A_i,B_i)$  are (up to isomorphisms) independent
on the choice of $T(\theta)$ as {they} result from {the} corresponding
eigenspaces of $A(\theta)$}.

\begin{theorem}\label{thm:decomposition}
  Let $\textbf{P} \subset \C$ be compact and contractible and let $X_n(\p)$ be a $M$-space.
  Moreover, let {$\Gamma_1, \dots, \Gamma_N$ be a pointwise disjoint continuous spectral
  decomposition of $(A,B) \in C_{n,n}(\p)\times X_{n,m}(\p)$.}
\begin{enumerate}
\item[(a)]
  If $(A,B)$ is ensemble reachable on $X_n(\p)$, then the associated subpairs $(A_i,B_i)$ are ensemble reachable on
  {$X_{n_i}(\p)$ for all $i=1,\dots,N$.}
\item[(b)]
  Conversely, if for all $i=1,\dots,N$ the associated subpairs $(A_i,B_i)$ are ensemble reachable on
  {$X_{n_i}(\p)$ and if the decomposition $\Gamma_1, \dots, \Gamma_N$ is additionally strictly
  disjoint such that the images $\Gamma_i(\p)$} do not separate the plane, then $(A,B)$ is ensemble reachable on $X_n(\p)$.
\end{enumerate}
\end{theorem} 
\begin{proof}
  (a): Assume that $(A,B)$ is ensemble reachable on $X_n(\p)$. Then {$\mathcal{M}_T$ and
    $\mathcal{M}_{T^{-1}}$ are well-defined bounded isomorphisms on the M-space $X_n(\p)$ and thus $(T^{-1}AT, T^{-1}B)$
    is also ensemble reachable on $X_n(\p)$. Hence $(A_i,B_i) = (\Pi_iT^{-1}AT\Pi_i^*, \Pi_iT^{-1}B)$ are obviously
    ensemble reachable on $X_{n_i}(\p)$ for $i=1,\dots,N$.} 

\medskip
\noindent
(b): Conversely, assume that $(A_i,B_i)$ are ensemble reachable on {$X_{n_i}(\p)$} for $i=1,\dots,N$.
Then, as $\Gamma_1, \dots, \Gamma_N$ is a {strictly} disjoint continuous spectral decomposition
such that $\C\setminus \Gamma_i(\p)$ is connected for every $i=1,\dots,N$, an application of Theorem~\ref{thm:parallel_multi}
yields that {the pair $(T^{-1}AT, T^{-1}B)$ is ensemble reachable on $X_n(\p)$ and likewise $(A,B)$
as $X_n(\p)$ is a M-space}.
\end{proof}

\medskip
The significance of Theorem~\ref{thm:decomposition} is that it allows to decompose the ensemble reachability problem
into several smaller problems according to {an underlying} spectral decomposition of the matrix
multiplication operator induced by $A(\theta)$. 

\begin{remark}
  {
    What can be done if the parameter space $\p$ is not contractible? -- If it can be decomposed into finitely many
    contractible and compact components $\p_1,...,\p_k$ {one can obviously combine Corollary \ref{cor:parallel_multi}} and Theorem~\ref{thm:decomposition}.}
\end{remark}

\section{Uniform ensemble reachability}\label{sec:unif}

In this section we focus on necessary and sufficient conditions for ensemble reachability on the Banach space of all
continuous functions, i.e. we consider $X_n(\p)=C_n(\p)$. We will put special emphasis on a clear distinction between
pointwise conditions (i.e.~conditions which have to be satisfied for all $\theta \in \p$) and ``global'' conditions which
are in general more difficult to test. We will first treat single input systems and the multi-input case afterwards. Note
that the results of this section extend previous results in
\cite{helmke2014uniform,li_tac_2016,scherlein2014open,schonlein2016controllability}, 
where the parameter set is assumed to be a compact real interval.

\subsection{Single-input parameter-dependent systems}\label{sec:unif-si}

\medskip

We begin this {subsection with conditions on the single-input pairs $(A,b) \in C_{n,n}(\p) \times C_{n}(\p)$
that are necessary for uniform ensemble reachability. In the following statement the necessary conditions~(d) and~(e) extend the necessary
conditions given in \cite[Lem.~1]{helmke2014uniform}. Recall that the set-valued map
$\operatorname{spec} A \colon \p \rightsquigarrow \C$ is termed {\it injective} if
$\sigma \big( A(\theta_1) \big) \cap \sigma \big( A(\theta_2) \big) = \emptyset$ for all $\theta_1\neq \theta_2 \in \p$.}


\begin{theorem}\label{thm:CC-SI-necessary-conditions}
Let $\p \subset \C$ be compact. Suppose $(A,b) \in C_{n,n}(\p) \times C_{n}(\p)$ is uniformly ensemble reachable. Then, the
following necessary conditions hold:
\begin{enumerate}
\item[(a)]
The pairs $\big(A(\theta),b(\theta))\big)$ are reachable for all $\theta \in \p$.
\item[(b)]
The eigenvalues of $A(\theta)$ have geometric multiplicity one for 
all $\theta \in \p$.
\item[(c)] The spectral map $\operatorname{spec} A$ is injective.
\item[(d)] If $\p$ is additionally contractible and $\operatorname{card} \p>1$, then the set $\{\theta \in \p \;|\; \operatorname{card} \sigma(A(\theta)) = n\}$ is open
and dense in $\p$.
\item[(e)]
The set $\p \subset \C$ has no interior points. 
\end{enumerate} 
\end{theorem}
\begin{proof}
  (a): Let $\theta \in \p$ be arbitrary {but fixed} and consider $\p_1:=\{\theta\}$. Then,
  a straightforward application of {Lemma~\ref{lem:resriction-property} and Remark \ref{rem:restriction-property},
    shows that the finite dimensional linear system associated with the fixed pair $(A(\theta),b(\theta))$} is approximately
  reachable and therefore reachable.

\medskip
(b): This follows immediately from~(a) together with the Hautus-Lemma \cite[Lemma~3.3.7]{sontag}.

\medskip

(c): The restriction property, cf.~Lemma~\ref{lem:resriction-property}, applied to 
$\p_1:=\{\theta,\theta'\}$ with $\theta \neq \theta' \in \p$ implies that the parallel connection 
\begin{align*}
\left(\begin{pmatrix} A(\theta) & 0\\ 0   & A(\theta') \end{pmatrix}, 
\begin{pmatrix} b(\theta) \\ b(\theta')\end{pmatrix} \right) 
\end{align*}
is reachable. Then, again a straightforward use of the Hautus Lemma \cite[Lemma~3.3.7]{sontag}
yields $\sigma(A(\theta)) \cap \sigma(A(\theta')) = \emptyset$.

\medskip
(d): Let {$\p_n := \{\theta \in \p \;|\; \operatorname{card} \sigma(A(\theta)) = n\}$}. Obviously,
due to Rouch\'{e}'s Theorem $\p_n$ is open. Therefore, it remains to show that $\p_n$ is dense in $\p$. Assume that 
$\p_n$ is not dense. Then there exists a non-empty open subset $U \subset \p \setminus \p_n$. Define 
{$m := \max_{\theta \in U} \operatorname{card} \sigma(A(\theta)) $}. By assumption one has $m < n$.
Again, by Rouch\'{e}'s Theorem, one can show that the non-empty set 
$U_m := \{\theta \in U \;|\; \operatorname{card} \sigma(A(\theta)) = m\}$ is open and that the algebraic multiplicities
of the eigenvalues are locally constant in $U_m$. Therefore, possibly by passing to a smaller open subset, 
we can assume that the algebraic multiplicities of the eigenvalues $\lambda_1(\theta), \dots, \lambda_m(\theta)$ are
constant in $U_m$. Moreover, by part (b) we already know that the geometric multiplicities of the eigenvalues are 
{equal to one. Hence, for all $\theta_0 \in U_m$ there exists $r > 0$ such that for all
$\theta \in K_{r}(\theta_0) := \{\theta \in \p \;|\; \|\theta-\theta_0\| \leq r\}$ one can simultaneously 
transform $A(\theta)$ into Jordan canonical from.}
This follows simply from continuity and the fact that $(A(\theta) - \lambda(\theta)I_n)^k$ has constant
rank\footnote{ Note that if 
$M(\theta_0) = \left(\begin{smallmatrix} M_{11}(\theta_0) & M_{12}(\theta_0)\\ 
M_{21}(\theta_0) & M_{22}(\theta_0)\end{smallmatrix}\right)$
has rank $k < n$ with $M_{11}(\theta_0)$ invertible, then there exists $r > 0$ such that $M_{11}(\theta)$
is invertible for all $\theta \in K_{r}(\theta_0)$ and due to the constant rank condition, one has 
$$
\left(\begin{smallmatrix} M_{11}(\theta) & M_{12}(\theta)\\ 
M_{21}(\theta) & M_{22}(\theta)\end{smallmatrix}\right) 
\left(\begin{smallmatrix} M_{11}(\theta)^{-1} & - M_{11}(\theta)^{-1} M_{12}(\theta)\\ 
0 & I_{n-m}\end{smallmatrix}\right)
= \left(\begin{smallmatrix} I_k & 0\\ 
* & 0\end{smallmatrix}\right)
$$
for all $\theta \in K_{r}(\theta)$. This immediately provides us with a basis $(b_1(\theta), \dots, b_{n-m}(\theta))$ 
of the kernel of $M(\theta)$ which depends continuously on $\theta$.} 
on $U_m$. Finally, by Theorem~\ref{thm:decomposition}~(a), {Lemma~\ref{lem:resriction-property}
and Remark \ref{rem:restriction-property},} it suffices to consider a single Jordan block $(J(\theta),b(\theta))$ on
$K_{r}(\theta_0)$. Since we assume $m < n$, there exists a Jordan block of size greater or equal than $2$ and thus
Proposition~\ref{prop:Jordan-norm-form} together with Remark~\ref{rem:Jordan-norm-form}~(b) yields the desired contradiction.

\medskip
(e): Assume that $\theta_0 \in \p$ is an interior point of $\p$. Moreover, according to part~(d) 
we can assume without loss of generality~that there exists $r > 0$ such that $K_r(\theta_0) \subset \p_n \subset \p$. Now, applying 
Proposition~\ref{prop:spectral-family} to the restriction of $(A,b)$ to $K_r(\theta_0)$ yields a continuous change of 
coordinates such that
\begin{equation}
T(\theta)^{-1} A(\theta) T(\theta) = 
\begin{pmatrix}
\lambda_1(\theta) & & 0\\
  & \ddots & \\
0 & & \lambda_n(\theta)
\end{pmatrix}
\quad 
\text{ and } \quad 
T(\theta)^{-1} b(\theta)= 
\begin{pmatrix}
1\\
 \vdots  \\
1
\end{pmatrix},
\end{equation}
where $\lambda_i(\theta)$ for {$i = 1, \dots, n$ are disjoint eigenvalue selections} of $A(\theta)$
on $K_r(\theta_0)=\p_0$. 
Note that pointwise reachability of $(A(\theta),b(\theta))$ guarantees that $T(\theta)^{-1} b(\theta)$ can be scaled
to $(1, \dots, 1)^\top$. By part~(c) the continuous curves $\lambda_i$ are injective on $K_r(\theta_0)$ for all $i=1,...,n$.
Thus, according to Brouwer's Invariance Theorem \cite[Ch.~V, Thm.~21.4]{newman1954planetop} the image $\lambda_i(K_r(\theta_0))$
has interior points in $\C$ for all $i=1,...,n$. Now, let us focus on $i=1$ and assume that $f_1 \in C(\p_0)$ is in the
closure of the reachable set of $(\lambda_1,1)$. Then, there exists a sequence of polynomials $p_n$ such that $p_n(\lambda_1(\theta))$
uniformly converges {to} $f_1(\theta)$ for all $\theta \in K_r(\theta_0)$ or, equivalently, such that $p_n(z)$ uniformly
converges {to} $f_1(\lambda^{-1}_1(z))$ for all $z \in \lambda_i(K_r(\theta_0))$. This, of course, implies, that 
$f_1(\lambda^{-1}_1(\cdot))$ is holomorphic in all interior points of $\lambda_i(K_r(\theta_0))$.
Thus, $(\lambda_1,1)$ is not uniform ensemble reachable on $K_r(\theta_0)$ and therefore $(A,b)$ is not uniformly ensemble
reachable {according to Theorem~\ref{thm:decomposition}~(a), Lemma~\ref{lem:resriction-property} and
Remark~\ref{rem:restriction-property}.}
\end{proof}

\medskip

The arguments used in the above proof further generalize to the following result. 

\begin{corollary}
Let $\p$ be homeomorphic to a compact subset of $\R^d$ with non-empty interior. Then, for $d \geq {2}$, the single
input pair $(A,b) \in C_{n,n}(\p) \times C_{n}(\p)$ is never uniformly ensemble reachable.
\end{corollary}
\begin{proof}
The case $d=2$ is shown in Theorem~\ref{thm:CC-SI-necessary-conditions}~(e). Assume without loss of generality $d = 3$
and $\p = K_1(0) \subset \R^3$, where $K_1(0)$ denotes the closed unit ball of $\R^3$. As in the proof of
Theorem~\ref{thm:CC-SI-necessary-conditions} one can show that there exists a possibly smaller open ball $K_r(\theta_0)$
such that $A$ restricted to $K_r(\theta_0)$ allows a {single-valued spectral decomposition
$\lambda_1(\theta), \dots, \lambda_n(\theta)$ with continuous and injective eigenvalue selection $\lambda_i(\theta)$.}
But this contradicts the Theorem of Borsuk-Ulam \cite[Ch.~IV, Thm~20.2]{Bredon_top_geo_1993} which {states}
that a continuous map from any sphere in $\mathbb{R}^3$ to $\R^2$ cannot be injective. 
\end{proof}

\medskip

{Next we consider sufficient conditions for uniform ensemble reachability. In the simplest case, i.e.~for
  scalar pairs $(a,b)$, we will see that the necessary conditions of Theorem \ref{thm:CC-SI-necessary-conditions} are also
  sufficient. In general, however, for non-scalar pairs $(A,b)$ additional assumptions have to be fulfilled to guarantee
  uniform ensemble reachability, cf.~Theorem~\ref{thm:suff-charakteristic}.}

\begin{proposition}\label{prop:scalar}
Let $\p\subset \mathbb{C}$ be a compact and contractible set with empty interior. Then the scalar pair
$(a,b) \in C(\p)\times C(\p)$ is uniformly ensemble reachable if and only if { $a: \p \to \C$} is injective and
$b(\theta) \neq 0$ for all $\theta \in \p$.
\end{proposition}
%
%
\begin{proof}
The necessity part follows from Theorem~\ref{thm:CC-SI-necessary-conditions}~(a) and~(c). To show sufficiency
we assume without loss of generality $b\equiv1$. 
Let $f \in C(\p,\C)$ and $\e>0$ be given. It suffices to prove that there is a polynomial $p$ such that 
\begin{align}\label{eq:scalar}
\sup_{\theta \in \p} |p(a(\theta)) - f(\theta)|  < \e.
\end{align}
Since $\p$ is compact, injectivity and continuity of $a$ imply that $a\colon \p \to a(\p) \subset\C$ is a
homeomorphism. Therefore, we conclude that $a(\p)$ is also contractible and thus by \cite[Prop.~4.2.8]{roe_winding_2015}
its complement $\C\setminus a(\p)$ is connected. Moreover, Brouwer's Theorem \cite[Ch.~V, Thm.~21.4]{newman1954planetop} shows that the interior
of $a(\p)$ with respect to $\C$ is empty. Then, by Mergelyan's Theorem \cite[Theorem~20.5]{rudin1987realcomplex}
there is a polynomial $p$ such that 
 \begin{align*}
\sup_{z\in a(\p)} |f(a^{-1}(z))-p(z)| < \e
\end{align*}
and thus \eqref{eq:scalar} follows.
\end{proof}

\medskip

Recall that for matrices {which depend continuously} on a parameter $\theta$ a continuous transformation
to the Jordan canonical form is in general not available, cf. \cite[$\S$~5.3 in Ch.~II]{kato}. However, the controllability 
canonical form of a pair $(A,b) \in C_{n,n}(\p)\times C_{n}(\p)$ which is pointwise reachable can be achieved continuously
as the subsequent statement shows.

\begin{lemma}[Canonical form]\label{lem:canonical} 
Let $\p \subset \C$ be compact and suppose that $\big(A(\theta),b(\theta)\big)$ is reachable for all
$\theta \in \p$. Then 
\begin{equation}
T(\theta) := \big(b(\theta), A(\theta)b(\theta), \dots, A^{n-1}(\theta)b(\theta)\big)
\end{equation}
is invertible for all $\theta$, its inverse $T(\theta)^{-1}$ depends continuously on $\theta$, and one has
\begin{equation}\label{eq:canonical}
A_c(\theta)=
T(\theta)^{-1}A(\theta)T(\theta) = 
\left(\begin{array}{ccccc}
   0 & \hdots & \hdots & 0 & a_0(\theta)\\
   1 & \ddots & & \vdots & a_1(\theta)\\
   0 &\ddots & \ddots & \vdots & \vdots\\
   \vdots & \ddots & \ddots & 0 & \vdots\\
   0 & \hdots & 0 & 1 & a_{n-1}(\theta)
\end{array}\right)
\quad\text{{and}}\quad 
T(\theta)^{-1}b(\theta)=e_1=
 \begin{pmatrix}
1\\
0\\
 \vdots \\
0
\end{pmatrix}\,,
 \end{equation}
where {$-a_k(\theta)$} are the coefficients of the characteristic polynomial of $A(\theta)$,
i.e.~{$\chi_{A(\theta)}(z) = z^n - (a_{n-1}(\theta)z^{n-1} + \cdots + a_0(\theta))$}. Moreover, the
pair $(A,b)$ is uniformly ensemble reachable if and only if $(A_c,e_1)$ is uniformly ensemble reachable.
\end{lemma}
\begin{proof}
By the Kalman rank condition \cite[Sec.~3.2, Thm.~3]{sontag} the pair $(A(\theta),b(\theta))$ is reachable if
and only if the matrix $T(\theta)$ has rank $n$, i.e. $T(\theta)$ is invertible. The continuity
of $\theta \mapsto T(\theta)^{-1}$ follows immediately from the continuity of the inversion map
{on $GL_n(\C)$ and \eqref{eq:canonical} is an immediate consequence of the Cayley-Hamilton Theorem,
  cf.~\cite[Lemma~5.1.3]{sontag}.}

To see the second claim, suppose that $(A_c,e_1)$ is uniformly ensemble reachable. Since $T(\cdot)$ is continuous
and $\mathbf{P}$ is a compact we define $c := \sup_{\theta\in\textbf{P}}\| T(\theta)\|<\infty$. So, for any
$p\in\mathbb{C}[z]$ and $f\in C_n(\textbf{P})$ one has
\begin{align*} {
 \sup_{\theta \in \p} \| p \left(A(\theta)\right)  \, b(\theta)  - f(\theta) \|               
  \leq c \sup_{\theta\in\textbf{P}}\|p(A_c(\theta))e_1-T(\theta)^{-1}f(\theta)\|}.
 \end{align*}
 As $T$ and $T^{-1}$ are continuous, we have {$\mathcal{M}_{T^{-1}}C_n(\textbf{P}) = C_n(\textbf{P})$}
 and thus,
uniform ensemble reachability of $(A_c,e_1)$ implies uniform ensemble reachability of $(A,b)$. 
The converse implication follows by the same reasoning. Alternatively, one can simply argue that $(A,b)$ and
$(A_c,e_1)$ are state space equivalent via the (continuously invertible) multiplication operator $\mathcal{M}_{T}$.   
\end{proof}

\medskip

The following statement extends Proposition \ref{prop:scalar} to non-scalar single input pairs. An additional 
assumption on the characteristic polynomials of $A(\theta)$ provides a sufficient condition for uniform ensemble
reachability which is a generalization of \cite[Thm.~2.1]{scherlein2014open}. But in contrast to the scalar case
this condition is no longer necessary.

\begin{theorem}\label{thm:suff-charakteristic}
  Let $\p \subset \C$ be compact and contractible and let the pair $(A,b) \in C_{n,n}(\p) \times C_{n}(\p)$ satisfy
  the necessary conditions of Theorem~\ref{thm:CC-SI-necessary-conditions}. Then, $(A,b)$ is uniformly ensemble
  reachable if the characteristic polynomials of $A(\theta)$ take the form
$z^n - (a_{n-1}z^{n-1} + \cdots +a_1z + a_0(\theta))$ for some $a_{n-1},...,a_1 \in \C$ and $a_0 \in {C}(\p)$. 
\end{theorem}
\begin{proof}
  By Lemma~\ref{lem:canonical} we can assume without loss of generality~that $(A(\theta),b(\theta))$ is in controllability
  form, i.e.
\begin{align*}
A(\theta)=  \left(\begin{array}{ccccc}
   0 & \hdots & \hdots & 0 & a_0(\theta)\\
   1 & \ddots & & \vdots & a_1 \\
   0 &\ddots & \ddots & \vdots & \vdots\\
   \vdots & \ddots & \ddots & 0 & \vdots\\
   0 & \hdots & 0 & 1 & a_{n-1} 
  \end{array}\right) \qquad b(\theta)  = e_1.
\end{align*}
%
%
%
%
To show the claim, we verify that for $\e>0$ and $f \in C_n(\mathbf{P})$ there is a polynomial $p$ so that
\begin{align*}
\sup_{\theta \in \mathbf{P}} \| p(A(\theta))\,e_1 - f(\theta) \| < \e.
\end{align*}
To this end, let $g(z):=z^n - (a_{n-1}z^{n-1}+ \dots +a_1 z)$ and define
\begin{align*}
 p(z):= \sum_{k=1}^n p_k\big(g(z)\big)z^{k-1},
\end{align*}
with $p_k \in \C[z]$ to be specified later. As $A(\theta)^k e_1 = e_{k+1}$ for $k=0,\dots,n-1$ and 
$g(A(\theta)) = a_0(\theta)I$ we obtain
\begin{align*}
p(A(\theta)) e_1
 =\sum_{k=1}^n p_k(g(A(\theta)))A(\theta)^{k-1}e_1
 =\sum_{k=1}^n p_k(g(A(\theta))e_k
 =\begin{pmatrix}
   p_1(a_0(\theta))\\
   \vdots\\
   p_n(a_0(\theta))
  \end{pmatrix}.
%
\end{align*}
Consequently, it remains to show that for appropriate choices of $p_k$ one has
\begin{align*}
\sup_{\theta \in \mathbf{P}} | p_k(a_0(\theta)) - f_k(\theta) | < \e \qquad \text{ for } k =1,...,n.
\end{align*}
The injectivity of the spectral map together with 
$\chi_{A(\theta)}(z)=z^n-(a_{n-1}z^{n-1}+ \dots + a_1 z + a_0(\theta))$ {implies that $a_0:\p \to \C$
  is one-to-one  and hence a homeomorphism}. Therefore, we conclude as in the proof of Prop.~\ref{prop:scalar} that
$\C\setminus a(\p)$ is connected and $a(\p)$ has empty interior. Thus, again by Mergelyan's Theorem
\cite[Theorem~20.5]{rudin1987realcomplex} there are polynomials $p_k$ such that 
\begin{align*}
\sup_{z\in a_0(\mathbf{P})}|p_k(z)-  f_k(a_0^{-1}(z))|<\varepsilon\qquad \text{ for } k =1,...,n.
\end{align*}
This shows the assertion.
\end{proof}

\medskip

{Now let us consider the following ensemble
 \begin{align}\label{eq:combined}
A(\theta)= \begin{pmatrix} 0 & -\theta^2& 0\\ 
                         1 &  0   & 0\\
                         0 &  0& \theta^2+1 
                          \end{pmatrix},  
\qquad b(\theta) =\begin{pmatrix} 1 \\ 0 \\ 1\end{pmatrix},\qquad \p =[0,1]
 \end{align}
 A brute-force analysis shows that \eqref{eq:combined} is uniformly ensemble reachable. This conclusion, however,
 cannot be drawn solely by Theorem~\ref{thm:suff-charakteristic}. Yet, combining Theorem~\ref{thm:parallel_multi}
 with Theorem~\ref{thm:suff-charakteristic} yields the desired result. In general, Theorems~\ref{thm:parallel_multi}
 and~\ref{thm:decomposition} allow for the following approach: First, determine a strictly disjoint continuous
 spectral decomposition of $A$ or equivalently of the corresponding matrix multiplication operator $\mathcal{M}_A$;
 then investigate the resulting subsystems (for instance via Theorem~\ref{thm:suff-charakteristic})} and finally glue
together the  individually pieces by Theorem~\ref{thm:parallel_multi} and~\ref{thm:decomposition}. This leads to the
following sufficient conditions.

\begin{theorem}\label{thm:CC-SI-sufficient-conditions}
  Let $\p \subset \C$ be compact and contractible and let the pair $(A,b) \in C_{n,n}(\p)\times C_n(\p)$ satisfy the
  necessary conditions of Theorem~\ref{thm:CC-SI-necessary-conditions}. Then $(A,b)$ is uniformly ensemble reachable
  if the following conditions are satisfied:
\begin{enumerate}
\item[(a)]
There exists a {strictly} disjoint continuous spectral decomposition
${\Gamma_1,...,\Gamma_k}$ with non-separating partial spectral sets $\Gamma_i(\p), {i=1,...,k}$.
\item[(b)]
The {characteristic polynomials of the associated subsystems $(A_i(\theta),b_i(\theta))$} take the form 
$z^{n_i} -\big( a_{i,n_i - 1}z^{n_i - 1} + \dots + a_{i,1} z + a_{i,0}(\theta)\big){, i=1,...,k}$ .
\end{enumerate}
\end{theorem}
\begin{proof}
By conditon~(a) we can apply Proposition~\ref{prop:spectral-family} and conclude the existence of a continuous
family $T(\theta)$ of invertible transformations such that $ A(\theta)$ becomes block-diagonal,
cf.~\eqref{eq:block-diagonal}. Then, by Theorem~\ref{thm:decomposition}~(b) it is sufficient to verify that
each associated subsystem is uniformly ensemble reachable. This, however, is guaranteed by assumptions (b) 
together with Theorem~\ref{thm:suff-charakteristic}.
\end{proof}

\medskip
The subsequent statement provides an extension of \cite[Theorem~1]{helmke2014uniform}, where $\p$ was assumed
to be a compact real interval. It follows immediately from Theorem~\ref{thm:CC-SI-sufficient-conditions} and
Lemma \ref{lem1:spectral-family}~(c).

\begin{corollary}\label{cor:scl_si}
Let $\p \subset \C$ be a compact, contractible and locally path-connected set with empty interior. Then, 
the pair $(A,b) \in C_{n,n}(\p)\times C_n(\p)$ is uniformly ensemble reachable if the following conditions are satisfied.
\begin{enumerate}
\item[(a)]
The pairs $(A(\theta),b(\theta))$ are reachable for all $\theta \in \p$.
\item[(b)]
The eigenvalues of $A(\theta)$ are simple for all $\theta \in \p$.
\item[(c)]
The spectral map is injective.
\end{enumerate}
\end{corollary}

\subsection{Multi-input parameter-dependent systems}

  \medskip
In this subsection we investigate parameter-dependent systems with more than one input. As in the single-input case we begin with necessary conditions for uniform ensemble reachability for pairs $(A,B) \in C_{n,n}(\p)\times  C_{n,m}(\p)$ {and recap the following result, cf.~\cite[Lemma~1]{helmke2014uniform}}. 

\begin{proposition}\label{thm:nec_multi}
Let $\p \subset \C$ be compact. If the pair $(A,B) \in C_{n,n}(\p)\times  C_{n,m}(\p)$  is uniformly ensemble reachability the following necessary conditions hold:
\begin{enumerate}
\item[(a)] The pairs $(A(\theta),B(\theta))$ are reachable for all $\theta \in \p$.
\item[(b)] The spectral map $\operatorname{spec} A$ is at most $m$-to-$1$, i.e.~for $s\geq m+1$ {and}
  pairwise distinct $\theta_1, \dots, \theta_s$ one has
\begin{align*}
\sigma (A(\theta_1))\cap \cdots \cap \sigma (A(\theta_s))=\emptyset.
\end{align*}
\end{enumerate}
\end{proposition}

\medskip
In case the pair $(A,B) \in {C_{n,n}(\p)\times  C_{n,m}(\p)} $ admits the special form $(\theta A,B)$,
where $A\in \C^{n \times n}$ and $B \in  \C^{n \times m}$, the above necessary and sufficient conditions can be stated
more precisely, cf.~\cite[{Theorem~1}]{li_tac_2016} and \cite[{Theorems~5 and~6}]{schonlein2016controllability} { in terms of the rank of the matrices $A$ and $B$. These conditions depend on whether $\p$ contains the origin or not. In particular, if $0 \in \p$ then any single input pair $(\theta A, b)$ is not uniformly ensemble reachable. In the following, based on Section~\ref{subsec:cascade-structures}, we investigate pairs $(A(\theta),B(\theta))$ with other special structures. We start with} upper triangular pairs of continuous matrix families $A{(\theta)}$ and $B{(\theta)}$, i.e.
\begin{align}\label{eq:lambda_lower_triangular}
A(\theta) = 
\begin{pmatrix} 
a_{11}(\theta) &  \cdots       &  a_{1n}(\theta)  \\ 
& \ddots &              \vdots      \\
&  & a_{nn}(\theta)
\end{pmatrix}   
\quad \text{ and } \quad
B(\theta) =
\begin{pmatrix} 
b_{11}(\theta) &    \cdots     &     b_{1n}(\theta) \\ 
& \ddots &           \vdots \\
 &  & b_{nn}(\theta)
\end{pmatrix}.
\end{align}
In this case we obtain the following sufficient condition for uniform ensemble reachability.

\begin{proposition}\label{prop:A_B_triangular}
Let $\textbf{P}\subset \mathbb{C}$ be a compact and contractible set with empty interior. The pair 
$(A,B)\in C_{n,n}(\p)\times  C_{n,n}(\p)$, defined in \eqref{eq:lambda_lower_triangular} is uniformly ensemble reachable
if $\operatorname{rank}B(\theta)=n$ for all $\theta \in \p$ and $a_{ii}$ is injective for all $i=1,...,n$.
\end{proposition}
\begin{proof}
By Proposition~\ref{prop:technical_MATCOM} it is sufficient to consider the diagonal pairs 
$(a_{ii}, b_{ii}) \in C(\p)\times C(\p)$. As  $\operatorname{rank}B(\theta)=n$ for all $\theta \in \p$ it 
follows that  $b_{ii}(\theta)\neq 0$ for all $\theta \in \p$ and for $i=1,...,n$. Since the functions
$a_{ii}$ are injective for $i=1,...,n$ we can apply  Proposition~\ref{prop:scalar} and the claim follows.
\end{proof}

\medskip
\begin{remark}\label{rem:upper-triangular-not-nec}
  The converse of the Propositions~\ref{prop:A_B_triangular} is false in general. To see this, let
  {$\p := [0,1]$}, $n=2$ and take continuous injective functions {$a_{11}$ and $a_{22}$} and 
$b = \binom{b_1}{b_2}\in C_2(\p)$ such that {$a_{11}(\p) \cap a_{22}(\p) = \emptyset $} and $b_1\neq 0$ and $b_2\neq 0$. Then, according to Corollary \ref{cor:scl_si}, the pair
\begin{align*}
\left(
\begin{pmatrix} 
a_{11} &  0  \\ 
0 & a_{22}
\end{pmatrix}   
,
\begin{pmatrix} 
b_{1} \\
b_{2}
\end{pmatrix}
\right)
\end{align*}
is uniformly ensemble reachable. Then, the pair 
\begin{align*}\left(
\begin{pmatrix} 
a_{11} &  0  \\ 
0 & a_{22}
\end{pmatrix}
,
\begin{pmatrix} 
0 &b_{1} \\
0& b_{2}
\end{pmatrix}\right)
\end{align*}
is clearly also uniformly ensemble reachable. However, the pair $(a_1,0)$ is obviously not uniformly  ensemble reachable. 
\end{remark}

\medskip
Next we consider a special case of \eqref{eq:lambda_lower_triangular} where $A$ takes the generalized Jordan structure 
\begin{align}\label{eq:Jordan-block}
J(\theta)= 
\begin{pmatrix} 
\lambda(\theta) &   \lambda_{12}(\theta)  &   \cdots     &     \lambda_{1n}(\theta)               \\ 
& \ddots &        \ddots     &        \vdots          \\
&  &         \ddots   &       \lambda_{n-1,n}(\theta)            \\
 &  & & \lambda(\theta)
\end{pmatrix} 
                                      \quad \text{ and } \quad
    B(\theta) \equiv B {\in \C^{n\times m}}\,,                                 
\end{align}
{i.e.~$B$ is constant.} The following result characterizes uniform ensemble reachability of the
pair $(J,B)$ and extends \cite[Proposition~3]{li_tac_2016} where, $\lambda (\theta)=\theta\, \lambda$ and
$\lambda_{i,i+1}(\theta)= \theta$ for all $i=1,...,n-1$ and $\lambda_{i,j}\equiv 0$ otherwise.

\medskip

\begin{proposition}\label{prop:Jordan-norm-form}
Let $\textbf{P}\subset \mathbb{C}$ be compact, contractible 
and let $\operatorname{card} \p > 1$. Then, the pair $(J,B)\in C_{n,n}(\p)\times {\C^{n\times m}}$, with
$m\geq n$ defined in \eqref{eq:Jordan-block} is uniformly ensemble reachable if and only if $\operatorname{rank}B=n$
and  $\lambda$ is injective.
\end{proposition}
\begin{proof}
We begin with the sufficiency part. Suppose that $\lambda$ is injective and $\operatorname{rank} B=n$. Then, 
without loss of generality let $B=I$ and the claim follows from Proposition~\ref{prop:A_B_triangular}.

\medskip
Conversely, let the pair $(J,B)$ be uniformly ensemble reachable. First, suppose that $\lambda$ is not injective, 
i.e.~there are $\theta_1 \neq \theta_2 \in \mathbf{P}$ such that {$\lambda(\theta_1)=\lambda(\theta_2)=:\lambda_0$}.
From the restriction property, cf.~Lemma \ref{lem:resriction-property}, it follows that 
the finite-dimensional parallel connection
 \begin{align*}
\left(\begin{pmatrix} J(\theta_1) & 0\\ 
   0   & J(\theta_2) 
\end{pmatrix}, \begin{pmatrix} B \\ B\end{pmatrix} \right) 
\end{align*}
is reachable. On the other hand, it is easy to see that {$\binom{e_n}{-e_n}$ in the othrogonal
  complement of the column space of the matrix
\begin{align*}
\begin{pmatrix} \lambda_0 I- J(\theta_1) & 0 & B \\ 0 &\lambda_0 I- J(\theta_2)& B \end{pmatrix}
\end{align*}
and thus its rank is less or equal $2n-1$. This} yields a contradiction to the classical Hautus Lemma
\cite[Lemma~3.3.7]{sontag}.

\medskip
\noindent
To see the necessity of the rank condition, we treat without loss of generality~the case $n=2$. Suppose that $\operatorname{rank} B < 2$.
Then, after a change of coordinates in the controls, we can assume
$$
B=\begin{pmatrix} b_1 & 0\\ b_2 & 0 \end{pmatrix}\,.
$$ 
Moreover, reachability of $(J(\theta),B)$ implies $\lambda_{12}(\theta) \neq 0$ for all $\theta \in \p$ and 
$b_2 \neq 0$. Thus we can make additionally the simplifying assumption $b_2 = 1$. 
Since $\lambda$ is continuous and injective and the set $\p$ is compact {and contractible the set
  $\lambda(\p)$ is path-connected. Moreover, due to the sssumption $\operatorname{card} \p > 1$,} we can choose $\gamma$
such that it has distinct endpoints denoted by $z_1\neq z_2$.
\medskip
\noindent
We first discuss the case where the length of $\gamma$ is finite, i.e. $L_\gamma < \infty$. Let $\e>0$ and  
define $f := \binom{c \lambda_{12}}{0} \in C_2(\p)$ with
{
\begin{align}\label{eq:def-c}
c > \e \left(2\, L_{\gamma}^{-1} + \tfrac{1+|b_1|}{ \min_{\theta \in \p} |\lambda_{12}(\theta)|}\right) + 1.
\end{align}}
As the pair $(J,B)$ is uniformly ensemble reachable there is a polynomial $p$ such that
\begin{align*}
\max_{\theta \in \p}\left\| p\big(J(\theta)\big) \, 
\begin{pmatrix} b_1 \\1 \end{pmatrix} - \,\begin{pmatrix} c \lambda_{12}(\theta) \\0 \end{pmatrix} \right\| <\e. 
\end{align*}
Thus, {for 
  \begin{align*}
    \Delta = \begin{pmatrix} \Delta_1 \\ \Delta_2 \end{pmatrix} :=
    p(J) \, \begin{pmatrix} b_1 \\1 \end{pmatrix} - \,\begin{pmatrix} c \,\lambda_{12}\\0 \end{pmatrix}
\end{align*}
we have $\|\Delta\|_\infty<\e$.}
By \cite[Ch.~6.1]{horn1991topics}, this results in
\begin{align*}
\begin{pmatrix} 
  p(\lambda(\theta))\,b_1 + p'(\lambda(\theta))\, \lambda_{12}(\theta) \\ p(\lambda(\theta))
\end{pmatrix}
= 
\begin{pmatrix} 
c\, \lambda_{12}(\theta)+ \Delta_1(\theta)\\ \Delta_2(\theta)
\end{pmatrix}\, 
\end{align*}
and therefore we obtain the estimates
\begin{align*}
|p(\lambda(\theta))|< \e 
\quad \text{ and } \quad 
\big|p'(\lambda(\theta)) - c \big| \leq 
\frac{ \big|\Delta_1(\theta) - b_1 \Delta_2(\theta)\big|}{\min_{\theta \in \p} |\lambda_{12}(\theta)|}
\leq \varepsilon \frac{1+|b_1|}{ \min_{\theta \in \p} |\lambda_{12}(\theta)|} .
\end{align*}
for all $\theta \in \p$. Hence, for all $z \in \lambda(\p)$, we have
\begin{align*}
|p(z)|< \e \quad \text{ and } \quad 
\big|p'(z) - c \big| \leq  \varepsilon \frac{1+|b_1|}{ \min_{\theta \in \p} |\lambda_{12}(\theta)|}.
\end{align*}
{Now, let $K \subset \C$ be any compact set which properly contains $\operatorname{tr}\gamma$.}
Since $p'$ is uniformly continuous on {$K$ there is a  $\delta>0$ such that
\begin{align}\label{eq:p'-unif_cont}
|p'(z)-p'(w)|<1 \quad \text{for all} \; z,w \in K \; \text{with} \; |z-w|< \delta.
\end{align}}
By standard arguments there is a sequence of polygons {$(\gamma_N:[0,1] \to K)_{N \in \N}$}
such that $\gamma_N(0)=z_1$ and $\gamma_N(1)=z_2$ for all $N\in \N$,
\begin{align*}
\|\gamma_N - \gamma\|_\infty < \delta
\quad \text{ and } \quad 
L_{\gamma_N} \to L_{\gamma} \text{ for } N \to \infty\,.
\end{align*}
{This together with \eqref{eq:p'-unif_cont} implies
  $ |p'(\gamma_N(t)) - p'(\gamma(t)) |< 1$ for all $t\in [0,1]$ and all $N \in \N$ and thus} one
has\footnote{{In general, $\gamma_N$ is only piecewise continuously differentiable
but for simplicity we assume that $\gamma_N$ is $C^1$.}}
\begin{align*}
2 \e &> \big|p(z_2) - p(z_1)\big| = \left| \int_{\gamma_N} p'(\xi) \di{\xi}\right|
= \left| \int_{0}^1 \Big( p'(\gamma_N(t)) - p'(\gamma(t)) + p'(\gamma(t)) - c +c \Big) \dot \gamma_N(t) \di{t} \right| \\
&\geq c\,L_{\gamma_N} -  \int_{0}^1 |p'(\gamma(t)) - c| \, |\dot \gamma_N(t)| \di{t} 
-  \int_{0}^1 |p'(\gamma_N(t)) - p'(\gamma(t))| \, |\dot \gamma_N(t)| \di{t} \\
& \geq \left( c -  \varepsilon \frac{1+|b_1|}{ \min_{\theta \in \p} |\lambda_{12}(\theta)|} - 1\right) \, L_{\gamma_N}\,.
\end{align*}
{Hence} taking the limit $N\to \infty$, we arrive at 
$2 \e \geq \left( c -  \varepsilon \frac{1+|b_1|}{ \min_{\theta \in \p} |\lambda_{12}(\theta)|} - 1\right) \, L_{\gamma} $. 
In particular, it follows
\begin{align*}
c \leq \e\,\left( 2\, L_{\gamma}^{-1} + \frac{1+|b_1|}{ \min_{\theta \in \p} |\lambda_{12}(\theta)| }\right)+ 1\,,
\end{align*}
which contradicts our choice of $c$ in \eqref{eq:def-c}.

\medskip
\noindent
Finally, assume that $L_\gamma =\infty$. Again, let $\e>0$ and define 
$f(\theta) := \binom{c \lambda_{12}(\theta)}{0} \in C_2(\p)$ with
\begin{align}\label{eq:def-c2}
  c > \e \left( \tfrac{1+|b_1|}{ \min_{\theta \in \p} \lambda_{12}(\theta)}\right) + 1\,.
\end{align}
Then, following the above arguments we obtain  
\begin{align*}
2 \e >\left( c -  \varepsilon \frac{1+|b_1|}{ \min_{\theta \in \p} |\lambda_{12}(\theta)|} - 1\right) \, L_{\gamma_N}
\end{align*}
and by {\eqref{eq:def-c2}} one has
$c -  \varepsilon \frac{1+|b_1|}{ \min_{\theta \in \p} |\lambda_{12}(\theta)|} - 1 >0$. 
This leads obviously to a contradiction 
since $L_{\gamma_N} \to \infty$ as $N\to \infty$. 
\end{proof}

\begin{remark}\label{rem:Jordan-norm-form}
\begin{enumerate}
\item[(a)]
Note that for a finite parameter set $\p := \{\theta_1, \dots, \theta_N\}$ the uniform ensemble reachability
problem of the pair $(J,B)$ boils down to a standard interpolation problem which can be solved exactly even
for single input systems whenever $\lambda(\theta_i) \neq \lambda(\theta_j)$ for $\theta_i \neq \theta_j$ and 
reachability of the individual systems $(J(\theta_i),B)$ is guaranteed.
\item[(b)]
In the last part of the above proof (concerning the necessity of rank condition on $B$) one could
easily allow $b_1$ to depend on $\theta$. This observation is important for an application of the above 
proposition in the proof of Theorem \ref{thm:CC-SI-necessary-conditions}. 
\end{enumerate}
\end{remark}

\begin{example}\label{ex:homogeneous_1}
 {Let $\p$ be a compact interval and consider the pair $(\theta A,B)$ with $A \in \C^{n\times n}$
    and $B\in \C^{n \times m}$. Note that this class of systems is extensively considered in \cite{li_tac_2016}
    and \cite{schonlein2016controllability}. In the following we sketch how to apply our previous results.
    Let $\lambda_1,...,\lambda_r$  denote the distinct eigenvalues of $A$ and let $T$
    be an invertible matrix such that
\begin{align*}
 T^{-1}AT = \begin{pmatrix}J_1 &  & \\ & \ddots & \\& & J_r \end{pmatrix} \quad T^{-1}B = \begin{pmatrix}B_1 \\ \vdots \\B_r \end{pmatrix},
\end{align*}
where $J_i$ is a generalized Jordan block, i.e. $J_i$ embraces all the Jordan blocks associated with the eigenvalue $\lambda_i$, for $i=1,...,r$.
\begin{enumerate}
  \item[(N)] Necessary for uniform ensemble reachability are that $(A,B)$ is reachable, $\operatorname{rank} A=n$ and $\operatorname{rank} B_i$ equals the size of $J_i$, $i=1,...,r$. In particular, $\operatorname{rank} B$ is necessarily at least the size of the largest generalized Jordan block of $A$.
  \item[(E)] If $0 \in \p$, then uniform ensemble reachability is equivalent to  $\operatorname{rank} A = \operatorname{rank} B=n$, cf. \cite[Thm.~1]{li_tac_2016} and \cite[Thm.~5]{schonlein2016controllability}
  \item[(S)] Sufficient for uniformly ensemble reachability is that in addition to (N) the following separation condition holds 
\begin{align*}
  \{\theta \lambda_i \, |\, \theta \in \p\} \cap   \{\theta \lambda_j \, |\, \theta \in \p\} = \emptyset
  \quad \text{ for all } i\neq j \in \{1,...,r\}.
\end{align*}  
 \end{enumerate}
We note that the above separation condition excludes the case that $0  \in \p$.  Assertion (N) follows from Proposition~\ref{thm:nec_multi}~(a) and~(b) and Proposition~\ref{prop:Jordan-norm-form}. The necessity part of (E) follows from Proposition~\ref{thm:nec_multi}~(a) and
     sufficiency is easily obtained by Propositions~\ref{prop:technical_MATCOM} and~\ref{prop:Jordan-norm-form}
     when assuming $A$ to be in Jordan canonical form and $B=\begin{pmatrix} I & 0 \end{pmatrix}$. Condition~(S) follows from the same reasoning.\\
     Finally, we note that for $0 \in \p$ a single-input pair $(\theta A,b)$ can never be uniformly ensemble reachable, unless $n=1$. Moreover, the application of Proposition~\ref{prop:scalar} yields that a scalar pair $(\theta a,b)$ is uniformly ensemble reachable if and only if $a\neq 0$ and $b \neq 0$ no matter whether $\p$ contains zero or not.
   }

\end{example}

\medskip
In order to derive sufficient conditions for uniform ensemble reachability for the case $0 \not \in \p$, we need to consider multi-input pairs that do not have {a} specific structure.  To this end, we recall the Hermite canonical form for a (parameter indenpendent) system $(A,B) \in \C^{n\times n} \times \C^{n\times m}$, cf.~\cite{sontag}. Let $b_i$ denote the $i$-th column of $B$. 
Select from left to right in the permuted Kalman matrix 
\begin{align}\label{eq:Hermite-list}
\begin{pmatrix}
 b_1 & Ab_1& \cdots A^{n-1}b_1&\cdots& b_m & Ab_m \cdots & A^{n-1}b_m
\end{pmatrix}
\end{align}
the first linear independent columns. Then, one obtains a list of basis vectors  
\begin{align*}
 b_1,...,A^{h_1-1}b_1,...,b_m,...,A^{h_m-1}b_m
\end{align*}
of the reachability subspace. The integers $h_1,...,h_m$ are called the {\it Hermite indices},
where $h_i:=0$ if the column $b_i$ has not been selected. One has $h_1+\cdots+h_m=n$ if and only if $(A,B)$ is reachable. 

{Now suppose} $(A,B)$ is reachable with Hermite indices $h_1\neq 0 , \cdots , h_k\neq 0$ and $ h_{k+1} =\cdots =h_m=0$. This is always achievable  by applying a suitable permutation. 
 Similar to Lemma~\ref{lem:canonical} (cf.~\cite[Section~6.4.6]{kailath1980}),
the invertible transformation
\begin{align*}\label{eq:Hermite-T}
T = (b_1,...,A^{h_1-1}b_1,...,b_k,...,A^{h_k-1}b_k )
\end{align*}
yields the Hermite canonical form 
\begin{equation}\label{eq:Hermite-canonical-form}
\tilde A(\theta) =
\begin{pmatrix}
\tilde A_{11}  & \cdots & \tilde A_{1k} \\
 & \ddots&  \vdots\\
0 & & \tilde A_{kk}
\end{pmatrix},
\quad
\tilde B(\theta) =
\begin{pmatrix}
\tilde b_{1} &  & 0 & * & * \\
 & \ddots & &*& * \\
0 & & \tilde b_{k}&*& * 
 \end{pmatrix}, 
\end{equation}
where the $k$ single-input subsystems $(\tilde A_{ii},\tilde b_i)\in \C^{n_i\times n_i}\times \C^{n_i}$, $i=1,...,k$ are reachable and in control canonical form. 

\medskip
\begin{theorem}\label{thm:CC-MI-sufficient-conditions}
Let $\p \subset \C$ be a compact. Then $(A,B) \in C_{n,n}(\p)\times C_{n,m}(\p)$ with $n \geq m$ is uniformly ensemble reachable if the following conditions are satisfied.
\begin{enumerate}
\item[(a)]
The  Hermite indices of $(A(\theta),B(\theta))$  do not depend on $\theta \in \textbf{P}$.
\item[(b)]
The corresponding subpairs $(\tilde A_{ii},\tilde b_i)$ are uniformly ensemble reachable for all $i=1,...,k$.
\end{enumerate}
\end{theorem}
\begin{proof}
By condition {(a)} the exists a continuous family $T(\theta)$ of invertible matrices which transforms the pair 
$(A,B)$ into a $\theta$-dependent Hermite canonical form \eqref{eq:Hermite-canonical-form}. Since each associated
subsystem $(\tilde A_{ii},\tilde b_i)$ is uniformly ensemble reachable the claim follows by applying Proposition~\ref{prop:technical_MATCOM} to the truncated system
\begin{equation*}
 \begin{pmatrix}
\tilde A_{11}{(\theta)}  & \cdots & \tilde A_{1k}{(\theta)} \\
 & \ddots&  \vdots\\
0 & & \tilde A_{kk}{(\theta)}
\end{pmatrix},
\quad
\begin{pmatrix}
\tilde b_{1}{(\theta)} &  & 0  \\
 & \ddots &  \\
0 & & \tilde b_{k} {(\theta)}
 \end{pmatrix}.
\end{equation*}
\end{proof}

\medskip

The proof of the latter statement shows that if the Hermite indices are constant the {multi-input} case can be {tackled by considering} a number of single-input systems. 
{Note that, different to the finite dimensional case the truncation of $\tilde B$ in the above proof can result in a massive loss of control options, cf. Proposition~\ref{prop:Jordan-norm-form} and Example~\ref{ex:m>n_reasonable}.
Nevertheless}, to each of the single-input systems {one} can apply the  sufficient conditions derived above. {Hence we} immediately get the following extension of \cite[Theorem~1]{helmke2014uniform}, where $\p$ was assumed to be a compact real interval.

\begin{corollary}\label{cor:scl_mi}
Let $\p \subset \C$ be a compact and contractible {and the notation as in \eqref{eq:Hermite-canonical-form}}. Then, $(A,B) \in C_{n,n}(\p)\times C_{n,m}(\p)$ is uniformly ensemble reachable if the following conditions are satisfied.
\begin{enumerate}
\item[(a)]
The pairs $(A(\theta),B(\theta))$ are reachable for all $\theta \in \p$.
{
\item[(b)]
The Hermite indices of $(A(\theta),B(\theta))$  do not depend on $\theta \in \textbf{P}$.
\item[(c)]
The spectral maps $\operatorname{spec} \tilde A_{ii}$ are injective.
\item[(d)]
The eigenvalues of $\tilde A_{ii}(\theta)$ are simple for all $\theta \in \p$.
}
\end{enumerate}
\end{corollary}

\medskip
Note that 
the Hermite indices of a pair $(A,B)$ are not invariant under input permutations. That is, {if} $P\in \C^{m\times m}$ be a permutation matrix, then the Hermite indices of $(A,B)$ and $(A, BP)$ do not coincide in general.
{This degree of freedom could be helpful sometimes.}

\medskip

%
We close this section with two examples. The first one illustrates that, in contrast to parameter-independent linear systems, it is of course reasonable to consider the case $m > n$ because $n$ is no longer the dimension
of the state space. 

\begin{example}\label{ex:m>n_reasonable}
Let $\p =[-1,1]$ and consider the pair $(a,B) \in C(\p)\times C_{1,2}(\p)$ defined by
 \begin{align*}
  a(\theta) = \theta^2 \qquad \text{ and } \qquad B(\theta) = \begin{pmatrix} 1 & \theta\end{pmatrix} .                                                                                       
 \end{align*}
The pair satisfies the necessary conditions in {Proposition~\ref{thm:nec_multi}} and has constant Hermite indices. To see that the pair is uniformly
ensemble reachable, let $f \in C(\p)$ and $\e>0$ be given. Then, we have to verify the existence of two
polynomials $p_1$ and $p_2$ such that
$| f(\theta) - p_1(\theta^2) - \theta \, p_2(\theta^2)|< \e $ for all $\theta \in [-1,1]$. By construction, we have
\[
 p_1(\theta^2) + \theta p_2(\theta^2) = c_0 + c_1 \theta + c_2 \theta^2 + \cdots + c_k \theta^k,
\]
where $c_{2l}$ denote the coefficients of $p_1$ and $c_{2l+1}$ denote the coefficients of $p_2$. Then, the claim 
{obviously follows by the} Weierstrass Approximation Theorem. We note that, since $a$ is not injective, we cannot apply Corollary~\ref{cor:scl_mi} to conclude the uniform ensemble reachability.
\end{example}

The next example demonstrates the role of the concept of contractibility in the previous results. 
More precisely, it shows that for the space of continuous functions $C(\p)$ 
{it might not be sufficient for uniform ensemble reachability to assume only the compactness of $\p$ (even if all other conditions
of Proposition~\ref{prop:scalar} or Theorem \ref{thm:suff-charakteristic}  are fulfilled).}


\begin{example}\label{ex:circle-not-uniform}
  {Let $\p =\partial \mathbb D$ and consider the pair $(a_0,b_0) \in C(\p)\times C(\p)$. defined by
\begin{align*}
a_0(\theta) = \theta \qquad \text{ and } \qquad b_0(\theta) = 1    .                         
\end{align*}}
The pair is not uniformly ensemble reachable. This can easily be seen using a continuous function
$f\colon \partial \mathbb{D} \to \C$ that has no analytic extension to $\mathbb{D}$, e.g. $f(z)= \tfrac{1}{z}$.
Suppose $(a,b)$ is uniformly ensemble reachable, then for $\e =1$ there is a polynomial $p$ such that 
$|p(z) -f(z)| <1$ for all $z \in \partial \mathbb D$. This implies
\begin{align*}
{|z\, p(z) - 1| <|z| =1   \qquad \text{ for all }  z \in \partial \mathbb D.}
\end{align*}
Hence the non-constant holomorphic function $q\colon {\mathbb{D}} \to \C$, $q(z):=zp(z)-1$ does not
attained its maximum {modulus} on the boundary (note $q(0) = 1$ and $|q(z)|<1$ for all $z \in \partial \mathbb{D}$)
which obviously contradicts the maximum modulus theorem 
\cite[Ch.~12]{rudin1987realcomplex}.

{The above arguments easily extended to the general case $(a,b) \in C(\p)\times C(\p)$ as follows:
  Suppose that $(a,b)$ is uniformly ensemble reachable. Then Theorem~\ref{thm:CC-SI-necessary-conditions} implies that
  we can assume without loss of generality that $a:\partial \mathbb{D} \to \mathbb{C}$ in one-to-one and $b \equiv 1$. Thus
  the image of $a:\partial \mathbb{D} \to \mathbb{C}$ is by definition a Jordan curve. Now, choose $z_0$ in the interior of
  $a(\partial \mathbb{D})$ (cf.~\cite[Theorem~4.3.5]{roe_winding_2015}) and define $f(z) := \frac{1}{z-z_0}$. Since $(a,b)$ is
  assumed to be uniformly ensemble reachable on $\partial \mathbb{D}$ for every $\varepsilon > 0$ one can find a polynomial
  $p$ such that $|p(z) -f(z)| < \varepsilon$ for all $z \in a(\partial \mathbb{D})$. For $\varepsilon > 0$ sufficiently small,
  this yields again a contradiction to the maximum modulus principle (applied to the holomorphic function $q(z):=(z-z_0)p(z)-1$ on
  the closure of interior of $a(\partial \mathbb{D})$).}
\end{example}

For scalar pairs it is well-known in complex analysis that the connectedness of $\C\setminus a(\p)$ is necessary for polynomial approximation, cf.~\cite[{Remark~}13.8]{rudin1987realcomplex}. This is exactly what is violated in Example~\ref{ex:circle-not-uniform}. Contractibility of $\p$ implies that $\C \setminus \p$ is connected , cf. \cite[Prop.~4.2.8]{roe_winding_2015} and thus, it excludes that $\C \setminus a(\p)$ is not connected.

%
\section{$L^q$-ensemble reachability}\label{sec:L_q}

In this section we focus on necessary and sufficient conditions for ensemble reachability on the separable
Banach spaces $L^q_n(\p)$, $1 \leq q < \infty$ consisting of all $L^q$-functions with values in $\C^n$. 
Throughout this section we assume that {$\mu$ is a \emph{regular\footnote{Note that the support
of $\mu$ can still be very ``small'' as a finite sum of point measures $\delta_i$, $i = 1,2,\dots$ is regular.
In this case $L^q_n(\p)$ behaves like a finite parameter set $\{\theta_i\;|\; i = 1,2,\dots,\}$ no matter
how ``large'' $\p$ seems to be.} (Borel) measure} (in the sense of \cite{lang_real_functional_1993}) on the
measure space $\p$ with $\sigma$-algebra $\Sigma(\p)$. For simplicity, we will omit the explicit reference to
$\mu$ whenever there is no reason for confusion.
We start with an auxiliary selection Lemma which is of its own interest.}

\begin{lemma}\label{lem:selection}
  Let $\p \in \C$ be compact and suppose the matrix-valued function {$M \colon \p \to \C^{n \times m}$}
  is measurable. Then, for the set-valued map $\theta \rightsquigarrow \ker M(\theta)$ there exists a $L^\infty$-selection
$v:\p \to \C^m$ such that $\|v(\theta)\| = 1$ whenever $\ker M(\theta) \neq \{0\}$.
\end{lemma}
\begin{proof}
By Lusin's Theorem \cite[Theorem~3.3, Chapter~IX]{lang_real_functional_1993} there are compact subsets $J_k$
of $\p$ such that $\mu\left( \p\setminus \bigcup_{k=1}^\infty J_k \right) =0$ and $\theta \mapsto M(\theta)$
is continuous on $J_k$ for every $k \in \N$. Now, consider the set-valued map $F\colon\p \rightsquigarrow \C^n$, 
\begin{align*}
F(\theta) = \begin{cases}
              \{0\} & \text{ if } \ker M(\theta)  = \{0\}\\
              \ker M(\theta)  \cap \overline{B}_1(0) & \text{ else. }
             \end{cases}
\end{align*}
Then, as $M$ is continuous on $J_k$ for every $k \in \N$ we have that the graph of $F|_{J_k}$ is closed and $F|_{J_k}$
is  bounded. Then, the arguments used in the proof of \cite[Theorem~A.7.3]{bressan2007introduction}
show that for every $k\in\N$ the lexicographical {selection\footnote{{The lexicographical
selection is given by $\xi_k(\theta) := \max F|_{J_k}(\theta)$, where the maximum is taken with respect to the lexicographical
order. Since the sets $F|_{J_k}(\theta)$ are compact, the maximum is guaranteed to exist for all $\theta$.}}}, denoted by $\theta \mapsto \xi_k(\theta)$, is
measurable on $J_k$. Then, $\xi\colon \bigcup_{k\in \N} J_k \to \C^n$, $\xi \big\vert_{{J_k}} (\theta) = \xi_k(\theta)$
is measurable and can be extended to a measurable function $v \colon \p \to \C^n$.
\end{proof}

\medskip
\noindent
We note that, even if the map $\theta \mapsto M(\theta)$ is continuous, the set-valued map $F$ does in general not allow
a continuous selection. This can be seen, for instance, using an example which is due to Rellich, cf.~\cite[II.\S~5.3]{kato}.

{We} call $\lambda:\p \to \C$ an \emph{$L^\infty$-eigenvalue selection} if $\lambda \in L^\infty(\p)$
satisfies $\lambda(\theta) \in \sigma(A(\theta))$ for almost all $\theta \in \p$ {and moreover} we say
that $\lambda$ is \emph{essentially injective} if there exists a set $\p_0 \subset \p$ of full measure such that
$\lambda|_{\p_0}$ is injective.

\subsection{{Single-input parameter-dependent systems}}

\medskip

{As in the uniform case we begin with necessary conditions for $L^q$-ensemble reachability.}

%
%

\begin{theorem}\label{thm:L^q-NEC}
Let $\p \subset \C$ be compact and suppose $(A,b) \in C_{n,n}(\p) \times L^q_{n}(\p)$ is $L^q$-ensemble reachable. 
Then $(A,b)$ satisfies the following necessary conditions:
\begin{enumerate}
\item[(a)]
The pairs $(A(\theta),b(\theta))$ are reachable for almost all $\theta\in \textbf{P}$.
\item[(b)]
The eigenvalues of $A(\theta)$ have geometric multiplicity one for almost all $\theta$.
\item[(c)]
Every $L^\infty$-eigenvalue selection {of $A$} is essentially injective.
\end{enumerate}
\end{theorem}
\begin{proof} (a) Suppose contrary that there is a set $\p' \subset \p$ with positive measure such that for
all $\theta \in \p'$ the pair $(A(\theta),{b}(\theta))$ is not reachable. Thus, for all $\theta \in \p'$ the rank
of the Kalman matrix $R_{(A,{b})}(\theta) = \left( {b}(\theta) \, A(\theta){b}(\theta)\, \cdots A(\theta)^{n-1}{b}(\theta) \right)$
is at most $n-1$. Hence, for each $\theta \in \p$ the dimension of the kernel of $R_{(A,{b})}(\theta)^\dagger$ is greater
or equal to one. Obviously, the map $\theta \mapsto R_{(A,{b})}(\theta)^\dagger$ is measurable. By Lemma~\ref{lem:selection}, 
there exists a $L^\infty$-function $\xi \colon \p \to \C^n$ such that $\xi(\theta)^\dagger R_{(A,{b})}(\theta) = 0$ for almost all 
$\theta \in \p$ and $\|\xi\| = 1$ for almost all $\theta \in \p'$. Consequently, the nonzero functional 
$L^q_n(\p) \ni f \mapsto \int_{\p} \xi(\theta)^\dagger f(\theta) \di{\mu}$ vanishes on the span of
$\{A(\cdot)^k b(\cdot) \; | \; k = 0,1,2,...\}$, which contradicts the $L^q$-ensemble reachability of $(A,b)$.

\medskip
\noindent

(b) This is an immediate consequence of (a).

\medskip
\noindent

(c) Let $\lambda\colon \p \to \C$ be a $L^\infty$-eigenvalue selection. Then, applying Lemma~\ref{lem:selection} to 
$A(\theta)^\dagger - \overline{\lambda(\theta)}I$, there is a nonzero $L^\infty$-function $v\colon \p \to \C^{n}$ such that on a set of full measure $\p_0$ one has 
\begin{align*}
 v(\theta)^\dagger A(\theta) = \overline{\lambda(\theta)} v(\theta)^\dagger \quad  \text{ and } \quad \|v(\theta)\|=1 
\quad \text{ for all } \theta \in \p_0,
\end{align*}
where $v(\theta)^{\dagger}$ denotes the complex-conjugate.
Now, consider the scalar system
\begin{align}\label{eq:L_infty_scalar}
\tfrac{\partial z}{\partial t}
(t,\theta)=\overline{\lambda(\theta)}z(t,\theta)+v(\theta)^\dagger b(\theta)\, u(t).
\end{align}
Note that if $\varphi(t,u,0)(\theta)$ is a solution to 
\begin{align*} 
\tfrac{\partial x}{\partial t}
(t,\theta)=A(\theta)x(t,\theta)+b(\theta)u(t)
\end{align*}
for the input $u \in L^1([0,T],\C)$, then $\psi(t,u,0)(\theta) := v(\theta)^\dagger \varphi(t,u,0)(\theta)$ is a solution
to \eqref{eq:L_infty_scalar} for the same input $u \in L^1([0,T],\C)$. 

Recall that if the scalar pair $(\lambda,v^\dagger\, b)$ is $L^q$-ensemble reachable then the multiplication operator $\mathcal{M}_{\lambda}$ is cyclic, cf.~page~4. Moreover, in \cite[Lemma~3.1]{seid1974cyclic} it is shown that $\mathcal{M}_{\lambda}$ is cyclic if and only if $\lambda$ is essentially injective\footnote{In \cite{seid1974cyclic} essentially injective is called essentially univalent.}.
Thus, the claim follows if we can show that the pair $(\lambda,v^\dagger\, b)$ is $L^q$-ensemble reachable. 
Let $\e>0$ and $g \in L^q(\p)$. Then $gv$ is obviously in $L_n^q(\p)$ and, since $(A,b)$ is $L^q$-ensemble 
reachable, there exists $T > 0$ and an input $u^* \in L^1([0,T],\C)$ such that
\begin{align*} 
\|\varphi(T,u^*,0) - gv\|_{L^q(\p)} < \e.
\end{align*}
Furthermore, one has
\begin{align*} 
\|\psi(T,u^*,0)  - g\|_{L^q(\p)} 
&= \left( \int_{\p} | \psi(T,u^*,0,\theta) - g(\theta)|^q \di{\mu} \right)^{\tfrac{1}{q}}
\\
&= \left( \int_{\p} | v(\theta)^\dagger \big( \varphi(T,u^*,0,\theta) - g(\theta) v(\theta)\big) |^q \di{\mu} \right)^{\tfrac{1}{q}}\\ 
& \leq \left( \int_{\p} \| v(\theta)\|^q \| \varphi(T,u^*,0)(\theta) - g(\theta) v(\theta)\|^q \di{\mu} \right)^{\tfrac{1}{q}}\\
&\leq \|\varphi(T,u^*,0) - gv\|_{L^q(\p)} < \e.
\end{align*}
This shows the assertion.
\end{proof}

\begin{remark}
  Condition~(c) in the previous result partially corresponds to condition~(c) of Theorem~\ref{thm:CC-SI-necessary-conditions}.
  However, {the actual analog one would like have is that} essential injectivity of the spectral map
  $\operatorname{spec} A \colon \theta \rightsquigarrow  \sigma \big( A(\theta) \big)$ is
  is necessary for $L^q$-ensemble reachability. We expect this to hold although the conjecture has resisted several
  attempts of proof. 
\end{remark}

Of course, every pair $(A,b) \in C_{n,n}(\p) \times C_{n}(\p)$ which is uniform ensemble reachable is also $L^q$-ensemble
reachable since the continuous functions are dense in $L^q_n(\p)$. The following example demonstrates that the converse
is false, i.e.~there are pairs $(A,b) \in C_{n,1}(\p)$ which are $L^q$-ensemble reachable but not uniform ensemble reachable.

\begin{example}
Let $\p := [0,1]$ be equipped with the Lebesgue measure and consider
\begin{align*}
A(\theta):=
\begin{pmatrix} \theta & 0\\ 0 & -\theta\end{pmatrix}
\quad\text{and}\quad
b(\theta) := \begin{pmatrix} 1 \\ 1\end{pmatrix}\,.
\end{align*}
The pair $(A,b)$ obviously satisfies the necessary conditions of Theorem~\ref{thm:L^q-NEC}. Moreover, Theorem~\ref{thm:suff-charakteristic} shows that $(A,b)$ is uniformly ensemble reachable over the parameter space $[c,1]$ for any $c > 0$, but Theorem~\ref{thm:CC-SI-necessary-conditions} implies that uniform ensemble reachability fails over $[0,1]$ as the pair $(A(0),b(0))$ is not reachable. However, we will see that $(A,b)$ is $L^q$-ensemble reachable over $[0,1]$ for $1 \leq q <\infty$.

\medskip
Let $f = \binom{f_1}{f_2} \in L^q_2(\p)$ and $\varepsilon > 0$. We have to show that there 
is a polynomial $p$ such that 
\begin{align*}
\|p - f_1\|_q < \e \qquad \text{ and } \qquad  \|p_- - f_2\|_q < \e
\end{align*}
with $p_-(\theta) := p(-\theta)$ for all $\theta \in [0,1]$. To this end, we choose continuous functions $g_1$
and $g_2$ on $\p$ such that  $\|f_i - g_i\|_q < \frac{\varepsilon}{2}$ and $g_i(0) = 0$ for $i=1,2$. Then, the 
continuous function $h \colon [-1,1] \to \C$ defined by 
\begin{align*}
 h(\theta) := \begin{cases}
              g_1(\theta) & \theta \in [0,1]\\
              g_2(-\theta) & \theta \in [-1,0]
             \end{cases}
\end{align*}
can approximated uniformly by a polynomial $p:[-1,1] \to \C$ such that $\|p-h\|_\infty < \frac{\varepsilon}{2}$. 
Consequently, we obtain 
\begin{align*}
\|p - f_1\|_q \leq  \|p - h\|_q + \|h - f_1\|_q \leq  \|p - h\|_q + \|g_1 - f_1\|_q  < \e 
\end{align*}
and
\begin{align*}
\|p_- - f_2\|_q \leq  \|p_- - h_-\|_q + \|h_- - f_2\|_q  \leq  \|p_- - h_-\|_q + \|g_2 - f_2\|_q < \e \,.
\end{align*}
\end{example}

In order to obtain sufficient conditions we will make use of the observation that a single-input pair 
$(A,b)\in C_{n,n} (\p) \times L_n^q(\p)$ is $L^q$-ensemble reachable if and only if the multiplication operator
\[ {\cal M}_A\colon L_{n}^q(\p) \to L_{n}^q(\p)\qquad {\cal M}_A\, f (\theta) = A(\theta)\,f(\theta)\]
is cyclic and $b$ is a cyclic vector for ${\cal M}_A$ (cf.~page~4). Similar to the uniform case, we consider scalar ensembles first.

\begin{proposition}\label{thm:Lq_scalar}
Let $\p \subset \C$ be compact. Then, the scalar pair $(a,b) \in C(\textbf{P})\times L^q(\p)$ is $L^q$-ensemble reachable for $q \in [1,\infty)$ if and only if $a:\p \to \C$ is essentially injective, $b(\theta) \neq 0$ for almost all $\theta \in \textbf{P}$, and 
\begin{equation}\label{eq:scalar-Lp}
  \inf_p \int_{\p} |p(a)b - \overline{a} b|^q \di{\mu} = 0,
\end{equation}
where $p$ runs through all complex polynomials.
\end{proposition}
\begin{proof}
The pair $(a,b)$ is $L^q$-ensemble reachable if and only if  
the multiplication operator ${\cal M}_a$ is cyclic and $b$ is a cyclic vector.
By \cite[Lem.~3.1]{seid1974cyclic} and \cite[Prop.~2.2]{ross2009common} the multiplication operator ${\cal M}_a$ is cyclic if and only if the function $a$ is essentially injective. {The necessity of the conditions $b(\theta) \neq 0$ for almost all $\theta \in \textbf{P}$ and \eqref{eq:scalar-Lp} follow from Theorem~\ref{thm:L^q-NEC}~(a) and the fact that $f := \overline{a} b$ belongs to $L_q(\p)$.} 

Next we show sufficiency. Let $\tilde \e >0$. Instead of $\mu$ we consider the measure $\mu_b$ induced by $|b|$, i.e.
\begin{equation*}
  {\mu_b(\Omega) := \int_{\p}|b|\di \mu \quad \text{for all}\; \Omega \in \Sigma(\p)\,.}
\end{equation*}
Then one has $f \in L_q(\p,\mu)$ iff $f/|b| \in L_q(\p,\mu_b)$ {and thus $L_q(\p,\mu)$ is isomorphic to $L_q(\p,\mu_b)$.}


Besides, for every $\e > 0$ we can find an open set {$N_{\e}$} such that $\mu(N_{\e}) < \e$ and $a$ restricted to the compact set $\p_{\e} := \p \setminus N_{\e}$ is injective and therefore a homeomorphism onto $a(\p_{\e})$. 
Now, we consider an arbitrary function $f \in L_q(\p,\mu_b)$. Since  $\mu_b$ is again regular, without loss of generality {we can assume} that $f$ is continuous.
Then, we define $g(z):= f(a^{-1}(z))$ for $z \in a(\p_{\e})$. Note that $a(\p_{\e})$ is a compact subset of $a(\p)$ and hence due to Tietze's Extension Theorem we can choose any continuous extension $\hat{g}$ of $g$ to $a(\p)$ with $\|\hat{g}\|_\infty = \|g\|_\infty {\leq \|f\|_\infty}$. The Stone-Weierstrass Theorem implies the existence of a polynomial $p$ in $z$ and $\bar{z}$ such that 
\begin{align*}
 \|p(z,\bar{z}) - \hat{g}(z)\|_{\infty} < \tilde \e \quad {\text{for all}\; z \in a(\p)\,.}
\end{align*}
This, together with Minkowski's inequality, leads to the following estimate
{
\begin{equation*}
\begin{split}
\int_{\p} &|p\big(a(\theta),\overline{a(\theta)}\big) - f(\theta)|^q \di \mu_b\\
& = \int_{N_{\e}} |p\big(a(\theta),\overline{a(\theta)}\big) - f(\theta)|^q \di \mu_b
+ \int_{\p_{\e}} |p\big(a(\theta),\overline{a(\theta)}\big) - f(\theta)|^q \di \mu_b\\
& \leq \left( \left( \int_{N_{\e}} |p\big(a(\theta),\overline{a(\theta)}\big)|^q \di \mu_b \right)^{\frac{1}{q}}
+ \left(\int_{N_{\e}} |f(\theta)|^q \di \mu_b \right)^{\frac{1}{q}}  \,  \right)^{{q}}
+ \int_{\p_{\e}} |p\big(a(\theta),\overline{a(\theta)}\big) - f(\theta)|^q \di \mu_b\\[2mm]
& \leq { \left( \tilde \e  + \,2 \|f\|_{\infty} \right)^q }\mu_b(N_{\e}) + \tilde \e^q \mu_b(\p_{\e}) \\&=
{\left(  \tilde \e + \,2\|f\|_{\infty}\right)^q }\int_{N_{\e}}|b(\theta)|\di \mu + \tilde \e^q \int_{\p_{\e}}|b(\theta)|\di \mu\\[2mm]
& \leq {\left(  \tilde \e + \,2\|f\|_{\infty}\right)^q }\int_{N_{\e}}|b(\theta)|\di \mu + {\tilde \e^q\, } \|b\|_{1,\mu}
\end{split}
\end{equation*}
Thus, choosing $\e$ sufficiently small we can guarantee that $\int_{N_{\e}}|b(\theta)|\di \mu < \tilde \e^q$ holds and hence
\begin{equation*}
  \| p(a,\overline{a}) - f\|_{L^q_n(\p,\mu_b)} = \left( \int_{\p} |p\big(a(\theta),\overline{a}(\theta)\big) - f(\theta)|^q \di \mu_b \right)^{\tfrac{1}{q}} \leq \tilde \e\,
  {\Big((\tilde \e + \,2\|f\|_{\infty})^q   + \|b\|_{1,\mu} \Big)^{\tfrac{1}{q}}  \,.}
\end{equation*}
}
Furthermore, \eqref{eq:scalar-Lp} implies that $\overline{a}$ is in the closure of the reachable set  $R(a,1)$ (with respect
to $\mu_b$). {It follows}\footnote{{If $p_n(a) \to \overline{a}$ with respect to $L^q(\p,\mu_b)$ then obviously $a^kp_n(a) \to a^k\overline{a}$ with respect to $L^q(\p,\mu_b)$. To see that $\overline{a}^2$ and all higher powers of $\overline{a}$ belongs to $R(a,1)$ choose $p_N$ such that $\|p_N(a) - \overline{a}\|_{1,\mu_b} < \e$. Then one has
$\|p_N(a)p_n(a) - \overline{a}^2\|_{1,\mu_b}
\leq \|p_N(a)\|_\infty \|p_n(a) - \overline{a}\|_{1,\mu_b} + \|\overline{a}\|_\infty\|p_N(a) - \overline{a}\|_{1,\mu_b} < 2 \e$
for $n$ sufficiently large.}} that for all $k,l\in \N$ the products $a^k \overline{a}^l$ are in the closure of the reachable
set of $(a,1)$ (with respect to $\mu_b$). Thus, for $p(z,\overline{z})$ and $\tilde \e>0$ there is a polynomial $\tilde p(z)$ such that
{
\begin{equation*}
\begin{split}
\| p(a,\overline{a}) - \tilde p(a)\|_{L^q_n(\p,\mu_b)} < \tilde \e.
\end{split}
\end{equation*}
}
Consequently, one has
{
\begin{equation*}
\begin{split}
\| \tilde p(a)-f \|_{L^q_n(\p,\mu_b)} &\leq  
\|\tilde p(a) - p(a,\overline{a}) \|_{L^q_n(\p,\mu_b)}  + \| p(a,\overline{a}) - f\|_{L^q_n(\p,\mu_b)} \\
&< \left(1+ \sqrt[q]{   (\tilde \e + \,2\|f\|_{\infty})^q   + \|b\|_{1,\mu} }\right) \tilde \e.
\end{split}
\end{equation*}
{This shows} that $L^q(\p,\mu_b)$ coincides with the closure of the reachable set  ${R}(a,1)$ and thus $L^q(\p,\mu)$ coincides with the closure of the reachable set ${R}(a,b)$. 
}
\end{proof}

\medskip

With the above pre-considerations, we are prepared to state our main result on $L^q$-ensemble reachability. Compared to the uniform case, the essential difference lies in the fact that -- at least to our knowledge -- there is no simply criterion which guarantees the existence of a {${L^\infty}$-transformation $T(\theta)$ which takes $A(\theta)$ into a ``nice'', for instence, diagonal form (except for selfadjoint families $A(\theta)$, cf.~Remark \ref{rem:Lq_general})}. Therefore, we unfortunately have to require the existence of such a transformation $T(\theta)$ in the following result.

\begin{theorem}\label{thm:Lq_general}
Let $\p \subset \C$ be compact. Then $(A,b) \in C_{n,n}(\textbf{P})\times L^q_{n,1}(\p)$ is $L^q$-ensemble reachable for $q \in [1,\infty)$ if the following conditions are satisfied:
\begin{enumerate}
\item[(a)]
The pairs $\big(A(\theta), b(\theta)\big)$ are reachable for almost all $\theta \in \textbf{P}$.
\item[(b)] 
The eigenvalues of $A(\theta)$ are simple for almost all $\theta \in \textbf{P}$ .
\item[(c)] The spectral map is essentially injective.
\item[(d)] 
There exist $\lambda_1,...,\lambda_n \in L^\infty(\p)$ and $T \in L^\infty_{n,n}(\p)$ such that the operator
${\cal M}_T \colon L^q_n(\p) \to L^q_n(\p)$ is bounded, invertible (with bounded inverse) and satisfies
\begin{equation*}
T(\theta)^{-1}A(\theta)T(\theta) = 
\begin{pmatrix} \lambda_1(\theta) & & \\ & \ddots & \\& & \lambda_n(\theta) \end{pmatrix}
\qquad \text {for almost all } \theta \in \p\,.
\end{equation*}
\item[(e)]
The approximation condition
\begin{equation*}
\inf_{p} \int_{\p} {\max_{k=1,...,n}} \big|p(\lambda_k)\hat{b}_k - \overline{\lambda}_k\hat{b}_k\big|^q \di \mu = 0
\end{equation*}
holds, where $p$ runs through all complex polynomials and $\hat{b}_k(\theta)$ denotes
the $k$-th component of $\hat b(\theta) := T(\theta)^{-1}b(\theta)$.
\end{enumerate}
\end{theorem}
\begin{proof}
Due to the assumption 
(d), it is sufficient to verify $L^q$-ensemble reachability of the pair
\begin{equation*}
\begin{pmatrix} \lambda_1(\theta) & & \\ & \ddots & \\ & & \lambda_n(\theta) \end{pmatrix}, 
 \quad 
 \hat b(\theta) := T(\theta)^{-1}b(\theta).
\end{equation*}
In order to apply Proposition~\ref{thm:Lq_scalar} we use the following transformation: Let $\p_n := \p \times \{1,...,n\}$ be the disjoint union of $n$ copies of $\p$ (each equipped with
the measure $\mu$) and define the isomorphism $\Phi:L^q_n(\p) \to L^q(\p_n)$ via
\begin{equation*}
\begin{split}
f = \begin{pmatrix} f_1 & \hdots & f_n\end{pmatrix}^\trans \mapsto \Phi(f), 
\quad \Phi(f) \, (\theta,k) = f_k(\theta)\,.
\end{split}
\end{equation*}
Moreover, we set
$\Lambda := \Phi (\lambda)$ for $\lambda = \begin{pmatrix} \lambda_1 & \cdots & \lambda_n\end{pmatrix}^\trans$
and consider the multiplication operator ${\cal M}_{\Lambda} : L^q(\p_n) \to L^q(\p_n)$ given by
\begin{equation*}
\begin{split}
g \mapsto \Lambda\,g, \quad 
\Lambda g (\theta,k) := \Lambda(\theta,k) \cdot g (\theta,k)    = \lambda_k(\theta)\cdot g (\theta,k).
\end{split}
\end{equation*}
Since $\Phi$ is an isomorphism the claim follows if the scalar multiplication operator ${\cal M}_{\Lambda}$
is cyclic with cyclic vector $\Phi(\hat{b})$. Conditions (a), (c) and (d) imply that ${\cal M}_{\Lambda}$ is
essential injective and that the entries $\hat b_k(\theta)$ are nonzero on a set of full measure. Moreover,
it is straightforward to show that condition (e) is equivalent to
\begin{equation*}
\inf_{p} \int_{\p_n}\big|p(\Lambda)\Phi(\hat{b}) - \overline{\Lambda}\Phi(\hat{b})\big|^q \di \mu = 0\,.
\end{equation*}
Hence it follows that we can apply {Proposition}~\ref{thm:Lq_scalar} and this shows the assertion.
\end{proof}

\begin{remark}\label{rem:Lq_general}
\begin{enumerate}
\item[(a)] 
If the eigenvalues $\lambda_1,...,\lambda_n \in C(\p)$ are real-valued, then condition~(e) is
automatically satisfied.
\item[(b)] 
 {In the case $A(\theta)$ is self-adjoint one can obviously choose $T(\theta)$ unitary and moreover, 
    if $\p$ is real and $A(\theta)$ additionally analytic this} often allows to construct a $L^\infty$-transformation as
  required in part (d) of the above Theorem, cf. \cite[Sec.~II.6]{kato}.
\item[(c)] 
  {For $p=2$ and self-adjoint bounded matrix-multiplication operator $\mathcal{M}_A$ it is well-known
    that $\mathcal{M}_A$ is in principle unitarily equivalent to a scalar multiplication operator $\mathcal{M}_a$ on
    $L^2([0,1],\tilde{\mu})$ for some appropriately chosen measure $\tilde{\mu}$, cf.~\cite[Theorem~4.1]{TianJorgensen}.
    Yet the construction of $\tilde{\mu}$ is not very explicit and resulting reachability criteria are difficult to test.}
\item[(d)] {It is easy to see that the scalar pair $(\theta,1)$ defined on $\p =\partial \mathbb{D}$
  is not $L^2$-ensemble reachable. However, it is shown in \cite{bram} that there exists a cyclic vector $b \neq 1$.
  Moreover, using a Theorem of Szeg\"o a complete characterization of cyclic vectors} is given in \cite[Cor.~2.8]{ross2009common}.
\end{enumerate}
\end{remark}

\subsection{{Multi-input parameter-dependent systems}}

\medskip

In this section we present {some} sufficient conditions for multi-input systems to be $L^q$-ensemble reachable. Consider {the} upper triangular pair of matrix families
\begin{align}\label{eq:lambda_lower_triangular_L^q}
A(\theta) = 
\begin{pmatrix} 
a_{11}(\theta) &  \cdots       &  a_{1n}(\theta)  \\ 
& \ddots &              \vdots      \\
&  & a_{nn}(\theta)
\end{pmatrix}   \in C_{n,n}(\p)
\quad \text{and} \quad
B(\theta) =
\begin{pmatrix} 
b_{11}(\theta) &    \cdots     &     b_{1n}(\theta) \\ 
& \ddots &           \vdots \\
 &  & b_{nn}(\theta)
\end{pmatrix}\in L^q_{n,n}(\p).
\end{align}
The following result relies on the cascade structure (Proposition~\ref{prop:technical_MATCOM}) and the $L^q$-ensemble reachability characterization given in Proposition~\ref{thm:Lq_scalar}. The proof is omitted as it is similar to the uniform case. 

\begin{proposition}\label{prop:Lq-MI-sufficient-conditions}
Let $\textbf{P}\subset \mathbb{C}$ be a compact. Then, the pair 
$(A,B)\in C_{n,n}(\p)\times  L^q_{n,n}(\p)$, defined in \eqref{eq:lambda_lower_triangular_L^q} is $L^q$-ensemble reachable
if {$a_{ii}$ is essentially injective, $b_{ii}(\theta) \neq 0$ for almost all $\theta \in \p$ and}  
\begin{equation}\label{eq:scalar-Lp_i}
  \inf_p \int_{\p} |p(a_{ii})b_{ii} - \overline{a_{ii}} b_{ii}|^q \di{\mu} = 0
\end{equation}
 {all $i=1,...,n$,} where $p$ runs through all complex polynomials.
%
%
%
%
\end{proposition}
We close this section by noting that an $L^q$-version of Theorem~\ref{thm:CC-MI-sufficient-conditions} is not straightforward. Even if the Hermite indices $h_1,...,h_m$ are constant for all parameters $\theta \in \p$ the application of the corresponding transformation
{
\begin{align*}
T(\theta) = (b_1(\theta),...,A^{h_1-1}(\theta)b_1(\theta),...,b_k(\theta),...,A^{h_k-1}(\theta)b_k(\theta) ) \in L^q_{n,n}(\p)
\end{align*}
yields submatrices $\tilde{A}_{ii}(\theta)$ which are not continuous in $\theta$ in general. Therefore, the previous theory does not apply. Other approaches to obtain sufficient conditions might be based on sufficient conditions for the matrix-multiplication
operator to be multicyclic on $L^q_n(\p)$ and the columns of $B$ to be cyclic vectors. {But} unfortunately,
we are not aware of such results. 
}

\section{Application to averaged reachability}\label{sec:output}

In applications such as e.g. cell biology or quantum systems, a frequently met task is to extract 
information of the system from average measurements. Throughout this section we consider $X_n(\p)$ to be either $C_n(\p)$ or $L_n^q(\p)$, for $q \in [1,\infty)$ and study linear ensembles
with outputs given by an average of the form
\begin{align*}
y=\int_{\textbf{P}}^{}C(\theta)f(\theta)\di{\mu},
\end{align*}
where $\mu$ is a regular Borel measure on $\p$ and for simplicity $C(\theta) \in \C^{p\times n}$ is a continuous matrix function. Hence, the output operator 
${\cal C}\colon X_n(\p) \to \C^p$,
\begin{align*}
{\cal C}{f} =  \int_{\p} C(\theta) f(\theta) \di{\mu}
\end{align*}
is bounded linear.

In the following we are interested in {pointwise testable conditions on the matrix triple $(A,B,C)$ such that arbitrary averages can be reached.} More precisely, we say that a triple $(A,B,C) \in C_{n,n}(\p)\times X_{n,m}(\p) \times C_{p,n}(\p)$ is
\textit{{averaged} reachable (from zero)}, if for any $y\in \mathbb C^p$ there exist {$T>0$ and  $u\in U(T)$} such that
\begin{align*}
{\cal C} \varphi(T,u,0)=\int_{\p} C(\theta) \varphi(T,u,0)(\theta) \di{{\mu}} = y.
\end{align*}
As the output space is finite dimensional the latter is equivalent to approximate {averaged} reachability,
i.e.~to the condition that for every $y \in \C^p$ and every $\e>0$ there exist {$T>0$ and  $u\in U(T)$} such that
\begin{align*}
\left\|{\cal C} \varphi(T,u,0) - y\right\|_{\C^p} <\e.
\end{align*}
Let $B=(b_1,...,b_m)$, then in the discrete-time case the output at time $T$ with input sequence
$u=(u{(0)},...,u{(T-1)})$ is given by
\begin{align*}
{\cal C} \varphi(T,u,0) =  \sum_{k=0}^{T-1}   \sum_{j=1}^m {u_{j}(T-1-k)} \int_{\p} C(\theta)A(\theta)^k b_j(\theta) \di{{\mu}}
\end{align*}
and the set of reachable outputs is given by
\begin{align*}
 \operatorname{span}\left\{ \int_{\textbf{P}}^{}C(\theta)A(\theta)^k b_j(\theta)\di{{\mu}}\,\, |\, \,{j=1,...,m,\,} k=0,1,2,... \right\}.
\end{align*}
{
Thus, in discrete-time a triple $(A,B,C) \in C_{n,n} (\p) \times X_{n,m}(\p)\times C_{p,n}(\p)$ is {averaged} reachable if and only if 
\begin{align}\label{eq:char_averaged}
\operatorname{span} \left\{  \int_{\textbf{P}}^{}C(\theta)A(\theta)^k b_j(\theta)\di{{\mu}}\, \, | \, \,   j=1,...,m\, ,\; k=0,1,2,... \right\} =\C^p. 
\end{align}
It is well-known that the latter characterization also holds for continuous-time systems, cf. \cite[Corollary~ 7.1.2]{triggiani75}. For the particular case $C(\theta) = I \in \R^{n \times n}$, yet under the weaker assumption that the matrix pair $(A(\theta),B(\theta))$ is only measurable, the same result can be found in \cite[Theorem~3]{zuazua2014averaged}. Note that, if $A(\theta)$ is only measurable the corresponding multiplication operator $\mathcal M_A$ as defined in \eqref{eq:def-operators-scalar} may be unbounded and therefore the results of Triggiani~\cite{triggiani75} which rely on the boundedness of the involved operators do not cover this case.}

{ Obviously, the if-and-only-if condition \eqref{eq:char_averaged}, which might be hard to verify,
  illustrates that averaged reachability is a subtle interplay between the matrices $A(\theta)$, $B(\theta)$,
  $C(\theta)$ and the measure $\mu$.
  Via the results of the previous sections we will derive some sufficient conditions for averaged reachability which are
  easy to test in terms of $(A(\theta),B(\theta),C(\theta))$.}
{To recap known necessary and sufficient conditions for averaged reachability let $C(\theta)=( c_{ij}(\theta) )_{i=1,..,p \atop j=1,...,n}$ and consider the functionals
\begin{align*}
 h_k\colon X_n(\p) \to \C ,\quad  h_k\, f = \int_\p  \sum_{j=1}^n c_{kj}(\theta) f_j(\theta) \di{\mu}, \qquad k=1,...,p.
\end{align*}
In \cite[Corollary~6.2]{triggiani75} it is shown that the triple $(A,B,C)$ is averaged reachable if $(A,B)$ is ensemble
reachable on $X_n(\p)$ and the functionals $h_1,...,h_p$ are linearly independent (in $X_n(\p)^*$). In the following we shall show an equivalent result, cf.~Prop.~\ref{thm:pointwise_outec}, which is more explicit
in terms of the matrices $C(\theta)$. To do so,
we define the output reachable set of $(A,B,C)$ by 
\begin{align*}
 R(A,B,C) = \left\{ \int_\p C(\theta) \varphi(t,u,0)(\theta) \di \mu \,\, |\,\,  {t\geq 0, \, \, u \in U(t)} \right\}
\end{align*} and recall that the \emph{support} of $\mu$ is given by $\operatorname{supp}(\mu):=\p \setminus \p_0$,
  where $\p_0$ is the largest open subset of $\p$ such that $\mu(\p_0) = 0$.}

\begin{lemma}\label{lem:output-reach}
Let $C \in C_{p,n}(\p)$ and assume that $(A,B) \in C_{n,n} (\p) \times X_{n,m}(\p)$ is ensemble reachable on $X_{n}(\p)$. If $\theta_0 \in \p$ belongs to the support of $\mu$ then  $\operatorname{Im}C(\theta_0) \subset R(A,B,C)$.
\end{lemma}
\begin{proof}
{Without loss of generality let $q=1$. Moreover, let $\theta_0 \in \operatorname{supp}(\mu)$ and $y = C(\theta_0)x$}. Then, by continuity, for $\e > 0$ there exists $\delta > 0$ such that  
$$
\|C(\theta) - C(\theta_0)\| < \e
$$
for all $\theta \in \p$ with $|\theta - \theta_0| < \delta$. Now, choose the step function
$$
g_{x,\delta} (\theta):= 
\begin{cases}
\mu(B_{\delta}(\theta_0) \cap \p)^{-1} x & \text{for } \theta \in B_{\delta}(\theta_0)\cap \p\,,\\
0 & \text{else.}
\end{cases}
$$
Note that $\mu(B_{\delta}(\theta_0) \cap \p)^{-1} > 0$ since $\theta_0 \in \operatorname{supp}(\mu)$. Then, due to our reachability assumption (in the uniform case use additionally Lusin's Theorem) there
exists an input $u$ such that $$\|g_{x,\delta} - \varphi(T,u,0)\|_1 < \e.$$
This yields the following estimate
\begin{equation}
\begin{split}
\Big\| \int_{\p} C(\theta)&\varphi(T,\theta,u) \di \mu - y \Big{\|} = 
\Big\| \int_{\p} C(\theta)\varphi(T,\theta,u) \di \mu - C(\theta_0)x \Big{\|}\\
& \leq \Big\| \int_{\p} C(\theta)\varphi(T,\theta,u) - C(\theta)g_{x,\delta}(\theta) \di \mu \Big{\|}
+ \Big\| \int_{\p} C(\theta)g_{x,\delta}(\theta) \di \mu - C(\theta_0)x \Big{\|}\\
& \leq \Big\| \int_{\p} C(\theta)\varphi(T,\theta,u) - C(\theta)g_{x,\delta}(\theta) \di \mu \Big{\|} \\
&\qquad \qquad + \Big\| \int_{B_\delta(\theta_0) \cap \p} C(\theta)g_{x,\delta}(\theta)  - \mu(B_{\delta}(\theta_0) \cap \p)^{-1} C(\theta_0)x \di \mu\Big{\|}\\
& \leq (\|C\|_\infty \,+ \|x\|)\, \e
\end{split}
\end{equation}
This shows {that} $y = C(\theta_0)x$ is approximately {averaged} reachable and due to finite dimension of the output space $y$ is also {averaged} reachable, i.e. $y\in R(A,B,C)$.
\end{proof}

\medskip

The latter statement is now used to show the following sufficient conditions for {averaged} reachability.  

\begin{proposition}\label{thm:pointwise_outec}
{The triple} $(A,B,C) \in C_{n,n} (\p) \times X_{n,m}(\p)\times C_{p,n}(\p)$ is {averaged} reachable if  
\begin{enumerate}
\item[$(a)$] $(A,B)$ is ensemble reachable on $X_n(\p)$.
\item[$(b)$] {There} are distinct $\theta_1,...,\theta_k{\in \operatorname{supp}\mu}$ such that $\operatorname{rank} \big( C(\theta_1)\,| \, \cdots \,|\, C(\theta_k)\big)=p$.
\end{enumerate}
{Moreover, condition~(b) is necessary for averaged reachability.}
\end{proposition}
\begin{proof}
 The output reachable set is a subspace of $\C^p$ and by Lemma~\ref{lem:output-reach}, we know that for each $i=1,...,k$ we have $\operatorname{Im} C(\theta_i) \subset R(A,B,C)$. The claim then follows from the equivalence of $\operatorname{Im} C(\theta_1) + \cdots +\operatorname{Im} C(\theta_k) = \C^p$ and $\operatorname{rank} \big( C(\theta_1)\,| \, \cdots \,|\, C(\theta_k)\big)=p$.
 
 \medskip
 
 {Next, it is straightforward to see that the linear functionals $h_1,...,h_p$ have to be
   linear independent for averaged reachability, cf.~\cite[Corollary~6.2]{triggiani75} . Moreover, by continuity of $C(\theta)$
   the functionals $h_1,...,h_p$ are linearly independent if and only if
\begin{align}\label{eq:equiv_functional_lin_indep_C}
(\alpha_1 \, \cdots \,\alpha_p) \, C(\theta) = 0 \quad \forall \, \theta \in \operatorname{supp} \mu \qquad \Longrightarrow \quad \alpha_1 = \cdots =\alpha_p=0.
\end{align}
Also, due to the finite dimensionality of the output space condition~(b) holds if and only if
\begin{align}\label{eq:image_C_spans_Cp}
\operatorname{span} \left\{ \operatorname{im} C(\theta) \, |\, \, \theta \in \operatorname{supp} \mu \right\}= \C^p.
\end{align}
Hence, the necessity of ~(b) follows immediately from the equivalence of \eqref{eq:equiv_functional_lin_indep_C} and \eqref{eq:image_C_spans_Cp}. }
\end{proof}

\medskip
{Note that for standard averaging operators, such as 
\begin{align*}
  {\cal C}f = \int_{\p} f(\theta) \di{\mu}\,,
\end{align*}
i.e. $C(\theta) =I$, condition~(b) in Proposition~\ref{thm:pointwise_outec} is automatically satisfied and hence the result boils down to the well-known and trivial fact that approximate reachability implies averaged reachability.
The lemma also shows that if condition~(b) is not satisfied, ensemble reachablility does not imply
averaged reachability. Further, by the results of Section~\ref{sec:unif} and Section~\ref{sec:L_q}, one can specify condition (a) of Proposition~\ref{thm:pointwise_outec} to get pointwise verifiable sufficient conditions for averaged reachability. Here, exemplarily, we only state the uniform case.
}

\begin{corollary}\label{cor:pointwise_outec}
Let $\p$ be compact and contractible. Then, $(A,B,C) \in C_{n,n} (\p) \times C_{n,m}(\p)\times C_{p,n}(\p)$ is {averaged} reachable if  
\begin{enumerate}
\item[$(a)$] $(A(\theta),B(\theta))$ is reachable for all $\theta\in \textbf{P}$.
\item[$(b)$] The Hermite indices of $(A(\theta),B(\theta))$  do not depend on $\theta \in \textbf{P}$.
\item[$(c)$] The spectral map is injective.
\item[$(d)$] For each $\theta\in \textbf{P}$, the eigenvalues of $A(\theta)$ are simple.
\item[$(e)$] There are distinct $\theta_1,...,\theta_k{\in \operatorname{supp}\mu}$ such that $\operatorname{rank} \big( C(\theta_1)\,| \, \cdots \,|\, C(\theta_k)\big)=p$.
\end{enumerate}
\end{corollary}

{The above conditions are quite strong and far from being necessary. The derivation of sharper
conditions for averaged reachability or, equivalently, for \eqref{eq:char_averaged} to hold is certainly desirable
but beyond the scope of this paper.} {Finally, we note that the results can naturaly be extended to output matrices $C(\theta)$ whose rows are in $X_n(\p)^*$.} 
\medskip

\section{Appendix}\label{sec:appendix}

\subsection{Proof of Lemma \ref{lem1:spectral-family}}\label{appendix:2}

\begin{proof}
(a): Assume without loss of generality $U = \p$ and define $k:= \max_{\theta \in \p}{\operatorname{card} \sigma(A(\theta)) }$. Choose $\theta_0 \in \p$ 
with $\sigma(A(\theta_0)) = \{\lambda_1, \dots, \lambda_k\}$ and $\lambda_i \neq \lambda_j$ for $i \neq j$.
Then there exist disjoint neighbourhoods $U_i$ with $\lambda_i \in U_i$ for $i = 1, \dots, k$ and therefore 
Rouch\'{e}'s Theorem \cite[Theorem~10.43~(b)]{rudin1987realcomplex} plus the maximality of $k$ guarantees
the existence of an {relatively} open neighbourhood $V$ of $\theta_0$ with $\operatorname{card} \sigma(A(\theta)) \cap U_i = 1$ for 
$i = 1, \dots, k$ and all $\theta \in V$. This allows to define a single-valued continuous spectral
decomposition on $V$.

(b): See \cite[{Chapter~II, \S~5, Theorem~5.2}]{kato}.

\medskip
\noindent
(c): The result should be well-known and follows from a straightforward application of the Lifting
Theorem, cf.~\cite[Ch.~III, Thm~4.1 \& Cor.~4.3]{Bredon_top_geo_1993}. Therefore, we only sketch the
necessary arguments.  Let $\C^{n \times n}_{\rm s}$ denote the set of all complex $n \times n$ matrices 
with simple eigenvalues and let $\operatorname{GL}_n(\C)$ be the set of all invertible complex 
$n \times n$ matrices. Moreover, let $\Delta_{n}(\C)$ consist of all complex diagonal matrices
and set $\Delta_{n,\rm s}(\C) := \Delta_{n}(\C) \cap \C^{n \times n}_{\rm s}$ and 
$\operatorname{D}_n(\C) := \Delta_{n}(\C) \cap \operatorname{GL}_n(\C)$. With these preliminaries we 
can construct a covering map \cite[Ch.~III, Def.~3.1]{Bredon_top_geo_1993}
\begin{equation*}
F:\widehat{\operatorname{GL}}_n(\C) \times \Delta_{n,\rm s}(\C) \to \C^{n \times n}_{\rm s}\,,
\quad\quad \big(T\operatorname{D}_n(\C),\Lambda\big) \mapsto T \Lambda T^{-1}\,,
\end{equation*}
where $\widehat{\operatorname{GL}}_n(\C) := \operatorname{GL}_n(\C)/\operatorname{D}_n(\C)$ denotes the
homogeneous space of all left cosets of $\operatorname{D}_n(\C)$. Obviously, $F$ is well-defined as
$\Lambda$ commutes with any another diagonal matrix. Moreover, since $\operatorname{GL}_n(\C)$ and 
$\Delta_{n,\rm s}(\C)$ are Hausdorff, path-connected and locally path-connected we conclude that 
$\widehat{\operatorname{GL}}_n(\C) \times \Delta_{n,\rm s}(\C)$ is  Hausdorff, path-connected
and locally path-connected, too. Finally, we have to show the existence of an elementary neighborhood 
$U$ for all $X \in \C^{n \times n}_{\rm s}$. To this end, we first compute $F^{-1}(X)$. W.l.o.g.~we can 
assume that $X$ is diagonal and thus it is straightforward to see that $X$ has $n!$ preimages consisting 
of pairs $(\Pi\operatorname{D}_n(\C), \Pi^{-1} X \Pi)$ where $\Pi$ denotes an arbitrary permutation matrix.
In order to see that each preimage has a neighborhood $U_\pi$ which is homeomorphically mapped to $U$
we can exploit the fact that $F$ is a local diffeomorphism with respect to the canonical manifold
structure of the homogeneous space $\widehat{\operatorname{GL}}_n(\C)$, \cite{Helgason,Warner}.
In particular for $\Pi = \operatorname{id}_n$ and $[\operatorname{id}_n] := \operatorname{D}_n(\C)$
we obtain
\begin{equation*}
\begin{split}
\operatorname{T}_{([\operatorname{id}_n],X)} F \; (P,H) 
& = \operatorname{T}_{[\operatorname{id}_n]} F(\,\cdot\,,X)\,P 
+ \operatorname{T}_{X} F([\operatorname{id}_n],\,\cdot\,)\,H\\
& = \frac{\di}{\di t}{\rm e}^{tP}X{\rm e}^{-tP}\big|_{t=0} + H = [P,X] + H\,,
\end{split}
\end{equation*}
where $P$ and $H$ are tangent vectors of $\widehat{\operatorname{GL}}_n(\C)$ at $[\operatorname{id}_n]$
and $\Delta_{n,\rm s}(\C)$ at $X$, respectively. Here, we can identify the tangent space of
$\Delta_{n,\rm s}(\C)$ at $X$ with set of all complex diagonal matrices $\Delta_{n}(\C)$ and the tangent 
space of $\widehat{\operatorname{GL}}_n(\C)$ at $[\operatorname{id}_n]$ with any complementary space 
$\mathfrak{p}$ of $\Delta_{n}(\C)$.
For instance, a convenient choice for $\mathfrak{p}$ is the 
orthogonal complement of $\Delta_{n}(\C)$ which consists of all  complex matrices 
which vanish on the diagonal. Thus it is straightforward to see that the tangent map of $F$ at $\big([\operatorname{id}_n],X \big)$
is invertible.  A similar computation yields
\begin{equation*}
\begin{split}
\operatorname{T}_{([\Pi],X)} F \; (P,H) 
& = \operatorname{T}_{[\Pi]} F(\,\cdot\,,X)\,P 
+ \operatorname{T}_{X} F([\Pi],\,\cdot\,)\,H\\
& = \frac{\di}{\di t}\Pi{\rm e}^{tP}X{\rm e}^{-tP}\Pi^{-1}\big|_{t=0} + \Pi H \Pi^{-1} = \Pi \big([P,X] +  H \big) \Pi^{-1}
\end{split}
\end{equation*}
and hence the tangent map of $F$ is always invertible. Therefore $F$ is a local diffeomorphism and this allows
to apply the Lifting Theorem \cite[Ch.~~III, Thm.~4.1 \& Cor.~4.3]{Bredon_top_geo_1993} to the map $A$, because 
$\p$ is by assumption locally path-connected and due to its contractibility also simply connected and path-connected. 
Hence there exists a lifting $\hat{A}$ of $A$ such that the following diagram commutes:

\begin{center}
 \begin{tikzpicture}
  \matrix (m) [matrix of math nodes,row sep=3em,column sep=4em,minimum width=2em]
  {
     {} & \widehat{\operatorname{GL}}_n(\C) \times \Delta_{n,\rm s}(\C)\\
     \p &  \C^{n \times n}_{\rm s}\\};
  \path[-stealth]
    (m-1-2)
    (m-2-1.east|-m-2-2) edge node [above] {}
            node [above] {$A$} (m-2-2)
    (m-1-2) edge node [right] {$F$}  (m-2-2)
   (m-2-1) edge node [above] {$\hat A$}  (m-1-2);
\end{tikzpicture}
\end{center}
Thus the map $\pi_2 \circ \hat{A}$, where $\pi_2$ denotes the projection onto the second component,
yields the desired continuous eigenvalue spectral decomposition.
\end{proof}



\subsection{Proof of Proposition~\ref{prop:spectral-family}}\label{appendix:3}

\begin{proof}
Since $\Gamma_1(\theta), \dots, \Gamma_k(\theta)$ 
are pointwise disjoint we can construct cycles $\Sigma_1(\theta), \dots, \Sigma_k(\theta)$ 
in the complex plane for all $\theta \in \p$ such that 
\[
\operatorname{ind}_{\Sigma_i(\theta)}(z) = 
\begin{cases}
1 & \text{for all}\; z \in \Gamma_i(\theta)\,,\\
0 & \text{for all}\; z \in \sigma\big(A(\theta)\big) \setminus \Gamma_i(\theta)\,,
\end{cases}
\]
where $\operatorname{ind}_{\Sigma_i(\theta)}(z)$ denotes the winding number of $z \in \mathbb{C}$ with respect
to $\Sigma_i(\theta)$. Hence, we can define the following spectral projections
\[
P_i(\theta) = \frac{1}{2 \pi { i} } \int_{\Sigma_i(\theta)} \big(zI-A(\theta)\big)^{-1} \di{z}.
\]
We claim that the map $\theta \mapsto P_i(\theta)$ is continuous and that the rank of 
$P_i(\theta)$ is constant with respect to $\theta$. To see that $\theta \mapsto P_i(\theta)$
is continuous we first note that the continuity of $\Gamma_i$ implies that the Hausdorff distance
between $\Gamma_i(\theta)$ and $\Gamma_i(\theta')$ tends to zero as $\theta'$ tends to $\theta$. 
Therefore, $\Gamma_i(\theta')$ is contained in
$\Omega_i(\theta) := \{z \in \mathbb{C} \;|\; \operatorname{ind}_{\Sigma_i(\theta)}(z) = 1\}$ 
for $\theta'$ sufficiently close to $\theta$ and thus, by Cauchy's Theorem 
\cite[Thm.~10.35]{rudin1987realcomplex}, one has 
\[
P_i(\theta') = \frac{1}{2 \pi { i} } \int_{\Sigma_i(\theta')} \big(zI-A(\theta')\big)^{-1} \di{z}
= \frac{1}{2 \pi { i} } \int_{\Sigma_i(\theta)} \big(zI-A(\theta')\big)^{-1} \di{z}.
\]
for $\theta'$ sufficiently close to  $\theta$. From the above representation of $P_i(\theta')$,
it follows that $\theta \mapsto P_i(\theta)$ is continuous. Next, we note that
$\sum_{i=1}^k \operatorname{rank} P_i(\theta) = n$ for all $\theta \in \p$ as 
$\Sigma_1(\theta) +\dots + \Sigma_k(\theta)$ yields a circle around the spectrum of $A(\theta)$. 
Moreover, by continuity with respect to $\theta$ one knows that the rank of $P_k(\theta')$ is 
greater or equal to the rank of $P_k(\theta)$ in a neighborhood of $\theta$. Since this holds for 
all $i = 1, \dots, k$ we conclude that the rank of $P_i(\theta)$ is locally constant with respect 
to $\theta$ and because of the connectedness of $\p$ it is globally constant. Finally, we can apply
a generalization of Dole\v{z}al's result \cite{dolevzal1964existence} which was obtained by Grasse 
\cite[Theorem~3.8]{grasse_laa_2004} and guarantees the existence of a continuous family of matrices 
$T_i(\theta) \in \mathbb{C}^{n \times n_i}$ with $n_i := \operatorname{rank} P_i(\theta)$ for all 
$i = 1, \dots, k$ such that the columns of $T_i(\theta)$ span the image of the $P_i(\theta)$ which is 
of course an $A(\theta)$-invariant subspace. Hence, there exist $A_i(\theta) \in \mathbb{C}^{n_k \times n_i}$
with
\[
A(\theta)T_i(\theta) = T_i(\theta)A_i(\theta)
\]
for $i = 1, \dots, k$. In particular, $A_i(\theta)$ are given by
$A_i(\theta) = \big(T_i(\theta)^* T_i(\theta)\big)^{-1}T_i(\theta)^* A(\theta)T_i(\theta) $ for 
$i = 1, \dots, k$. Then stacking all $T_i(\theta)$ together, i.e.~setting
$T(\theta) := \big(T_1(\theta) \,|\, \dots \,|\, T_k(\theta)\big) \in \mathbb{C}^{n \times n}$, yields
the desired result
\[
T(\theta)^{-1} A(\theta) T(\theta) = 
\begin{pmatrix}
A_1(\theta) & & 0\\
 & \ddots & \\
0 & & A_k(\theta)
\end{pmatrix}\,.
\]
The stated spectral condition follows simply from the fact that $P_i(\theta)$ is by construction
the spectral projection onto all generalized eigenspaces whose eigenvalues are surrounded by 
$\Gamma_i(\theta)$. 
\end{proof}



\subsection{Proofs of the Lemmata in the proof of Theorem~\ref{thm:parallel_multi}}\label{appendix:1}

Let $\gamma$ denote a closed (piecewise) $C^1$-path in the plane and let $\operatorname{tr}\gamma$ denote its trace. Then, 
\begin{align*}
 \operatorname{ind}_{\gamma}(z) := \tfrac{1}{2\pi i} \int_{\gamma} \frac{1}{\xi-z} \di{\xi}\,,
\quad z\in \C\setminus\operatorname{tr}{\gamma}
\end{align*}
denotes the winding number of $z$ with respect to $\gamma$. A closed polygon $\tau=[p_1\,p_2\,\cdots\, p_k\, p_1]$ 
composed of finitely many horizontal or vertical segments $[p_1\,p_2]$,$[p_2\,p_3]$,..., $[p_k\,p_1]$ is called a \emph{grid polygon} if 
there exists a not necessarily regular grid $G \subset \C$ of horizontal or vertical lines such that all 
vertices $p_1,...,p_k$ are pairwise distinct adjacent grid point of $G$. As shown in \cite[\S~4.2 in Chapter~12]{remmert2013classical}\footnote{We note that in
\cite{remmert2013classical} the more general term step polygon is used instead of grid polygon. But the proof is actually only given for grid polygon.} 
every grid polygon divides the complex plane into exactly two disjoint domains
\begin{align}\label{eq:jordan-grid-polygon}
 \C \setminus \operatorname{tr}\gamma = \operatorname{ext}\gamma \cup \operatorname{int}\gamma \,,
\end{align}
with
\begin{align*}
 \operatorname{ext} \gamma := \{ z\in \C\setminus \operatorname{tr}\gamma \, | \,  \operatorname{ind}_{\gamma}(z) =  0\}
\end{align*}
and
\begin{align*}
 \operatorname{int} \gamma := \{ z\in \C\setminus\operatorname{tr}\gamma\, | \,  \operatorname{ind}_{\gamma}(z) =  1\}
\quad \text{ or } \quad  
\operatorname{int} \gamma := \{ z\in \C\setminus\operatorname{tr}\gamma\, | \,  \operatorname{ind}_{\gamma}(z) =  -1\}
\end{align*}
depending on the orientation of $\gamma$. Moreover, as remarked in \cite[\S~4.2 in Chapter~12]{remmert2013classical} one can show that for every grid polygon $\gamma$ one has
$\overline{\operatorname{int}\gamma} = \operatorname{int}\gamma \cup \operatorname{tr}\gamma$.

\begin{theorem}[Circuit Theorem]\label{thm:Circuit}
Let $K$ be a compact subset of the non-empty open set $\Omega \subset \C$.
\begin{enumerate}
 \item[(a)] Then, for every connected subset $K_0 \subset K$ there is a grid polygon $\tau$ in 
$\Omega\setminus K$ such that $\operatorname{ind}_{\tau}(K_0)=1$. 
 \item[(b)] If $K$ is additionally non-separating and connected and if $\C\setminus \Omega$ has 
only finitely many bounded connected components, then there is a grid polygon $\tau$ in $\Omega\setminus K$
such that $\operatorname{ind}_{\tau}(K)=1$ and $\overline{\operatorname{int} \tau} \subset \Omega$.
 \end{enumerate}
\end{theorem}
\begin{proof}
Part~(a) is shown in \cite[\S~4.2 in Chapter~12]{remmert2013classical}.

(b) For simplicity, we treat only the case of one bounded connected component because all arguments easily 
extent to finitely many bounded connected components. Thus, let $V$ denote the bounded connected component
of $\C \setminus \Omega$. As $V$ is compact there is square $Q \subset \C$ that properly contains $K$ and $V$. 
Now choose $v \in V$ and $q \in \partial Q$. Since $K$ does not separate the plane there is path $\gamma$ 
connecting $v$ and $q$ such that $\operatorname{tr}\gamma \cap K = \emptyset$. Then let $\gamma_\infty$ denote any
continuous continuation of $\gamma$ which connects $q$ with $\infty$ (for instance, one can choose
a straight half line parallel to the real or imaginary axis depending on the location of $q$ on $\partial Q$). 
Then, the application of part~(a) to $\Omega_0 = \Omega \setminus \gamma_\infty$ yields a grid polygon 
$\tau$ in $\Omega_0 \setminus K$  with $\operatorname{ind}_{\tau}(K)=1$. 

Finally, we have to show $\operatorname{int}\tau \subset \Omega$. To this end we consider any 
$ z \not \in \Omega$. Then $z$ lies in an unbounded  connected component of $\C\setminus \Omega_0$ because
any unbounded connected component of $\C\setminus \Omega$ belongs to an unbounded connected component of
$\C\setminus \Omega_0$ and $V$ is by construction part of an unbounded connected component of $\C\setminus \Omega_0$. 
Therefore, we conclude $\operatorname{ind}_{\tau}(z)=0$ and thus $z \not\in \operatorname{int}\tau$.
Hence it follows $\operatorname{int} \tau \subset \Omega$ and, as $\tau$ itself is also contained 
in $\Omega$, we obtain $\overline{\operatorname{int} \tau} \subset \Omega$.
\end{proof}

\begin{lemma}\label{lem:union-not-sep}
Let $C_1,...,C_N$ be finitely many, pairwise disjoint, compact and connected subsets in $\C$. 
Then the union $C_1 \cup \cdots \cup C_N$ is non-separating, i.e. $\C\setminus (C_1 \cup \cdots \cup C_N)$
is connected, if and only if each $C_i$ is non-separating for $i = 1, \dots, N$.
\end{lemma}
\begin{proof}
``$\Longrightarrow$'': Assume without loss of generality~that $C_1$ is separating. Then it is easy to show that 
$C_1 \cup \cdots \cup C_N$ is also separating. Note that the connected components of 
$\C \setminus C_1$ are open and hence due to the disjointness condition the union of 
$C_2, \dots, C_N$ cannot cover any connected component of $\C \setminus C_1$.

\medskip
``$\Longleftarrow$'':
First we show the case $N=2$. The general case follows by induction. To this end, we pick points $p$ and 
$p'$ in $\C\setminus (C_1 \cup C_2)$ and show that there is a path $\gamma$ in $\C\setminus (C_1 \cup C_2)$
connecting them. Since $C_1$ and $C_2$ are compact there is a square $Q$ that properly contains $C_1$ and $C_2$. 
Obviously, is suffices to show that one can connect $p$ with an arbitrary (but fixed) $q$ on the boundary
of $Q$ without intersecting $C_1$ and $C_2$. Since $\C \setminus C_1$ is connected there is a path 
$\gamma_{pq}$ connecting $p$ and $q$ such that $\operatorname{tr}\gamma_{pq} \cap C_1 = \emptyset$.
If $\operatorname{tr}\gamma_{pq} \cap C_2 = \emptyset$ holds we are done.
Therefore, we consider the case 
\begin{align*}
\operatorname{tr}\gamma_{pq} \cap C_2 \neq \emptyset\,.
\end{align*}
Our goal is to modify the path $\gamma_{pq}$ to a path, say $ \tau_{pq}$, such that 
$\operatorname{tr}\tau_{pq} \cap C_2 = \emptyset $. For this, we apply part~(b) of the Circuit Theorem~\ref{thm:Circuit} to the compact connected set $C_2$ lying in $Q \setminus (C_1 \cup \{p\})$  and conclude the existence of a closed grid polygon $\tau_{C_2}$ in $Q \setminus (C_1 \cup \{p\})$ such that 
\begin{align*}
\overline{\operatorname{int} \tau}_{C_2} \subset Q \setminus (C_1 \cup \{p\})\,.
\end{align*}
Next define $I:=\{t \in [0,1] \, |\, \gamma_{pq}(t) \in \overline{\operatorname{int} \tau}_{C_2}\}$ and let 
$t_- := \min I$ and $t_+ := \max I$. Then, the path $\tau_{pq}$ is obtained by joining $p$ and 
$\gamma_{pq}(t_-)$ along $\gamma_{pq}$, and $\gamma_{pq}(t_-)$ and $\gamma_{pq}(t_+)$ along $\tau_{C_2}$, 
and $\gamma^1_{pq}(t_+)$ and $q$ along again $\gamma_{pq}$. 

The case $N > 2$ can easily be handled by induction. By assumption, the union of $C_1, \dots, C_{N-1}$
is non-separating and thus we can find path $\gamma_{pq}$ which connects $p$ and $q$ without intersecting
$C_1 \cup \cdots \cup C_{N-1}$. If $\operatorname{tr}\gamma_{pq} \cap C_N \neq \emptyset$ we can proceed as above to
obtain a  modified  path $\tau_{pq}$ which connects $p$ and $q$ without intersecting 
$C_1 \cup \cdots \cup C_{N}$.
\end{proof}

\begin{lemma}\label{lem:extension}
Let $C_1,...,C_N$ be finitely many, pairwise disjoint, compact, connected and non-separating subsets
in  $\C$. Then, there are pairwise disjoint, compact, connected sets $K_1,...,K_N$ such that each 
$C_i$ is properly contained in $K_i$ for $i=1,\dots,N$ and the union of $K_1,...,K_N$ does not separate 
the plane.
\end{lemma}
\begin{proof}
First we show the claim for $N=2$. By applying part~(b) of the Circuit Theorem \ref{thm:Circuit} to
$\Omega := \C\setminus C_2$ and $K:=C_1$ and get a grid polygon $\tau_1$ such that $\operatorname{int}\tau_1$
properly contains $C_1$ and satisfies $\overline{\operatorname{int}\tau_1 } \cap C_2 = \emptyset$. Thus,
we define $K_1:=\overline{\operatorname{int}\tau_1}$ and apply again part~(b) of the Circuit Theorem
\ref{thm:Circuit} to $\Omega := \C \setminus K_1$ and $K:=C_2$. This yields a grid polygon $\tau_2$
such that $C_2 \subset \operatorname{int}(\tau_2)$ and 
$\overline{\operatorname{int}\tau_2 } \cap K_1 = \emptyset$. Then setting
$K_2:=\overline{\operatorname{int}\tau_2}$ and taking \eqref{eq:jordan-grid-polygon} into account it
follows that $K_1$ and $K_2$ do not separate the plane and therefore Lemma~\ref{lem:union-not-sep}
shows that $K_1\cup K_2$ does not separate the plane.

\medskip

For $N>2$ we can repeat the above construction with the obvious modification that in the $k$-th step 
one has to apply part~(b) of the Circuit Theorem \ref{thm:Circuit} to the sets
$\Omega := \C \setminus (K_1 \cup \cdots \cup K_{k-1} \cup C_{k+1} \cup \cdots \cup C_{N}) $ and $K:=C_k$.

\end{proof}

\medskip
The next result is trivial, but given that this defines a building block in our construction methods 
for ensemble reachability we state it separately for future reference. 

\begin{lemma}\label{lem:runge-extension}
Let $K_1$ and $K_2$ be disjoint compact non-separating sets {with finitely many components}. Then, the function
$h\colon K_1\cup K_2 \to \C$ defined by $h(z)=1$ for all $z\in K_1$ and $h(z)=0$ for all $z \in K_2$
can be uniformly approximated by polynomials.
\end{lemma}
\begin{proof}
Choose arbitrary open disjoint neighborhoods $U$ and $V$ of $K_1$ and $K_2$, respectively, such that the function $h$ is analytic on $U \cup V$. Then,
the assertion follows from Runge's Approximation Theorem, cf. \cite[Theorem~13.7]{rudin1987realcomplex}.
\end{proof}

\section*{Acknowledgment}\label{sec:acknowledgment}
This research was supported by the German Research Foundation (DFG) within the grants HE 1858/14-1 and SCHO 1780/1-1. We also thank the referees for their helpful comments and valuable suggestions.


%

\end{document}